\documentclass[11pt]{amsart}
\usepackage{}
\usepackage{amssymb}
\usepackage{amsfonts}
\usepackage{mathrsfs}
\usepackage{latexsym}
\usepackage{graphicx}
\usepackage{amscd,amssymb,amsmath,amsbsy,amsthm,amsfonts}
\usepackage[all]{xy}
\usepackage[backref,colorlinks,plainpages,urlcolor=blue]{hyperref}%
\usepackage{verbatim}
\usepackage{enumerate}

\usepackage [english]{babel}
\usepackage [autostyle, english = american]{csquotes}
\MakeOuterQuote{"}

\newcommand {\Omit}[1]{}

\usepackage{tikz-cd}

\usepackage{tikz}
\usetikzlibrary{automata,positioning}

\usetikzlibrary{arrows,calc}
\tikzset{
>=stealth',
help lines/.style={dashed, thick},
axis/.style={<->},
important line/.style={thick},
connection/.style={thick, dotted},
}
\usetikzlibrary{patterns}
\newlength{\hatchspread}
\newlength{\hatchthickness}
\tikzset{hatchspread/.code={\setlength{\hatchspread}{#1}},
         hatchthickness/.code={\setlength{\hatchthickness}{#1}}}
\tikzset{hatchspread=3pt,
         hatchthickness=0.4pt}
\pgfdeclarepatternformonly[\hatchspread,\hatchthickness]
   {custom north west lines}
   {\pgfqpoint{-2\hatchthickness}{-2\hatchthickness}}
   {\pgfqpoint{\dimexpr\hatchspread+2\hatchthickness}{\dimexpr\hatchspread+2\hatchthickness}}
   {\pgfqpoint{\hatchspread}{\hatchspread}}
   {
    \pgfsetlinewidth{\hatchthickness}
    \pgfpathmoveto{\pgfqpoint{0pt}{\hatchspread}}
    \pgfpathlineto{\pgfqpoint{\dimexpr\hatchspread+0.15pt}{-0.15pt}}
    \pgfusepath{stroke}
   }

\newcommand{\nc}{\newcommand}
\nc{\rnc}{\renewcommand}
\nc{\bb}[1]{{\mathbb #1}}
\nc{\bbA}{\bb{A}}\nc{\bbB}{\bb{B}}\nc{\bbC}{\bb{C}}\nc{\bbD}{\bb{D}}
\nc{\bbE}{\bb{E}}\nc{\bbF}{\bb{F}}\nc{\bbG}{\bb{G}}\nc{\bbH}{\bb{H}}
\nc{\bbI}{\bb{I}}\nc{\bbJ}{\bb{J}}\nc{\bbK}{\bb{K}}\nc{\bbL}{\bb{L}}
\nc{\bbM}{\bb{M}}\nc{\bbN}{\bb{N}}\nc{\bbO}{\bb{O}}\nc{\bbP}{\bb{P}}
\nc{\bbQ}{\bb{Q}}\nc{\bbR}{\bb{R}}\nc{\bbS}{\bb{S}}\nc{\bbT}{\bb{T}}
\nc{\bbU}{\bb{U}}\nc{\bbV}{\bb{V}}\nc{\bbW}{\bb{W}}\nc{\bbX}{\bb{X}}
\nc{\bbY}{\bb{Y}}\nc{\bbZ}{\bb{Z}}
\nc{\mbf}[1]{{\mathbf #1}}
\nc{\bfA}{\mbf{A}}\nc{\bfB}{\mbf{B}}\nc{\bfC}{\mbf{C}}\nc{\bfD}{\mbf{D}}
\nc{\bfE}{\mbf{E}}\nc{\bfF}{\mbf{F}}\nc{\bfG}{\mbf{G}}\nc{\bfH}{\mbf{H}}
\nc{\bfI}{\mbf{I}}\nc{\bfJ}{\mbf{J}}\nc{\bfK}{\mbf{K}}\nc{\bfL}{\mbf{L}}
\nc{\bfM}{\mbf{M}}\nc{\bfN}{\mbf{N}}\nc{\bfO}{\mbf{O}}\nc{\bfP}{\mbf{P}}
\nc{\bfQ}{\mbf{Q}}\nc{\bfR}{\mbf{R}}\nc{\bfS}{\mbf{S}}\nc{\bfT}{\mbf{T}}
\nc{\bfU}{\mbf{U}}\nc{\bfV}{\mbf{V}}\nc{\bfW}{\mbf{W}}\nc{\bfX}{\mbf{X}}
\nc{\bfY}{\mbf{Y}}\nc{\bfZ}{\mbf{Z}}
\nc{\bfa}{\mbf{a}}\nc{\bfb}{\mbf{b}}\nc{\bfc}{\mbf{c}}\nc{\bfd}{\mbf{d}}
\nc{\bfe}{\mbf{e}}\nc{\bff}{\mbf{f}}\nc{\bfg}{\mbf{g}}\nc{\bfh}{\mbf{h}}
\nc{\bfi}{\mbf{i}}\nc{\bfj}{\mbf{j}}\nc{\bfk}{\mbf{k}}\nc{\bfl}{\mbf{l}}
\nc{\bfm}{\mbf{m}}\nc{\bfn}{\mbf{n}}\nc{\bfo}{\mbf{o}}\nc{\bfp}{\mbf{p}}
\nc{\bfq}{\mbf{q}}\nc{\bfr}{\mbf{r}}\nc{\bfs}{\mbf{s}}\nc{\bft}{\mbf{t}}
\nc{\bfu}{\mbf{u}}\nc{\bfv}{\mbf{v}}\nc{\bfw}{\mbf{w}}\nc{\bfx}{\mbf{x}}
\nc{\bfy}{\mbf{y}}\nc{\bfz}{\mbf{z}}

\newcommand{\op}{\text{op}}

\nc{\mcal}[1]{{\mathcal #1}}
\nc{\calA}{\mcal{A}}\nc{\calB}{\mcal{B}}\nc{\calC}{\mcal{C}}\nc{\calD}{\mcal{D}}
\nc{\calE}{\mcal{E}} \nc{\calF}{\mcal{F}}\nc{\calG}{\mcal{G}}\nc{\calH}{\mcal{H}}
\nc{\calI}{\mcal{I}}\nc{\calJ}{\mcal{J}}\nc{\calK}{\mcal{K}}\nc{\calL}{\mcal{L}}
\nc{\calM}{\mcal{M}}\nc{\calN}{\mcal{N}}\nc{\calO}{\mcal{O}}\nc{\calP}{\mcal{P}}
\nc{\calQ}{\mcal{Q}}\nc{\calR}{\mcal{R}}\nc{\calS}{\mcal{S}}\nc{\calT}{\mcal{T}}
\nc{\calU}{\mcal{U}}\nc{\calV}{\mcal{V}}\nc{\calW}{\mcal{W}}\nc{\calX}{\mcal{X}}
\nc{\calY}{\mcal{Y}}\nc{\calZ}{\mcal{Z}}
\nc{\fA}{\frak{A}}\nc{\fB}{\frak{B}}\nc{\fC}{\frak{C}} \nc{\fD}{\frak{D}}
\nc{\fE}{\frak{E}}\nc{\fF}{\frak{F}}\nc{\fG}{\frak{G}}\nc{\fH}{\frak{H}}
\nc{\fI}{\frak{I}}\nc{\fJ}{\frak{J}}\nc{\fK}{\frak{K}}\nc{\fL}{\frak{L}}
\nc{\fM}{\frak{M}}\nc{\fN}{\frak{N}}\nc{\fO}{\frak{O}}\nc{\fP}{\frak{P}}
\nc{\fQ}{\frak{Q}}\nc{\fR}{\frak{R}}\nc{\fS}{\frak{S}}\nc{\fT}{\frak{T}}
\nc{\fU}{\frak{U}}\nc{\fV}{\frak{V}}\nc{\fW}{\frak{W}}\nc{\fX}{\frak{X}}
\nc{\fY}{\frak{Y}}\nc{\fZ}{\frak{Z}}
\nc{\fa}{\frak{a}}\nc{\fb}{\frak{b}}\nc{\fc}{\frak{c}} \nc{\fd}{\frak{d}}
\nc{\fe}{\frak{e}}\nc{\fFf}{\frak{f}}\nc{\fg}{\frak{g}}\nc{\fh}{\frak{h}}
\nc{\fri}{\frak{i}}\nc{\fj}{\frak{j}}\nc{\fk}{\frak{k}}\nc{\fl}{\frak{l}}
\nc{\fm}{\frak{m}}\nc{\fn}{\frak{n}}\nc{\fo}{\frak{o}}\nc{\fp}{\frak{p}}
\nc{\fq}{\frak{q}}\nc{\fr}{\frak{r}}\nc{\fs}{\frak{s}}\nc{\ft}{\frak{t}}
\nc{\fu}{\frak{u}}\nc{\fv}{\frak{v}}\nc{\fw}{\frak{w}}\nc{\fx}{\frak{x}}
\nc{\fy}{\frak{y}}\nc{\fz}{\frak{z}}

\newtheorem{theorem}{Theorem}[section]
\newtheorem{lemma}[theorem]{Lemma}
\newtheorem{corollary}[theorem]{Corollary}
\newtheorem{prop}[theorem]{Proposition}

\newtheorem{assumption}[theorem]{Assumption}

\theoremstyle{definition}
\newtheorem{definition}[theorem]{Definition}
\newtheorem{example}[theorem]{Example}
\newtheorem{remark}[theorem]{Remark}

\newtheorem{thm}{Theorem}

 \DeclareMathOperator{\id}{id}
\DeclareMathOperator{\Image}{Im}

\DeclareMathOperator{\modu}{mod} 
 \DeclareMathOperator{\GL}{GL}
\DeclareMathOperator{\Hom}{{Hom}}

\DeclareMathOperator{\Hilb}{{Hilb}}

 \DeclareMathOperator{\tr}{tr}

\DeclareMathOperator{\Grass}{Grass} \DeclareMathOperator{\End}{End}

\DeclareMathOperator{\Gm}{\bbG_m}

   \DeclareMathOperator{\coind}{CoInd}
\newcommand{\sph}{\fs}
    
\DeclareMathOperator{\sp1}{sp}
  \DeclareMathOperator{\triv}{triv}
  \DeclareMathOperator{\prepr}{p}
   \DeclareMathOperator{\aux}{aux}
      \DeclareMathOperator{\crit}{c}
   \DeclareMathOperator{\sign}{sign}

\DeclareMathOperator{\cl}{cl}

\DeclareMathOperator{\Crit}{Crit}

\DeclareMathOperator{\CH}{CH}

\DeclareMathOperator{\MO}{MO}

\DeclareMathOperator{\inc}{in}
\DeclareMathOperator{\out}{out}

\newcommand{\sE}{\epsilon}

\DeclareMathOperator{\Rep}{Rep}
\DeclareMathOperator{\Res}{Res}
\DeclareMathOperator{\Ind}{Ind}
\DeclareMathOperator{\Sh}{Sh}

\newcommand{\inj}{\hookrightarrow}

\newcommand{\pt}{\text{pt}}
\newcommand{\Aff}{\bbA}

\newcommand{\C}{\bbC}

\newcommand{\Q}{\bbQ}
\newcommand{\N}{\bbN}

\DeclareMathOperator{\fac}{fac}
\DeclareMathOperator{\pr}{pr}

\newcommand{\loc}{loc}
\DeclareMathOperator{\BM}{BM}

\newcount\cols
{\catcode`,=\active\catcode`|=\active
 \gdef\Young(#1){\hbox{$\vcenter
 {\mathcode`,="8000\mathcode`|="8000
  \def,{\global\advance\cols by 1 &}%
  \def|{\cr
        \multispan{\the\cols}\hrulefill\cr
        &\global\cols=2 }%
  \offinterlineskip\everycr{}\tabskip=0pt
  \dimen0=\ht\strutbox \advance\dimen0 by \dp\strutbox
  \halign
   {\vrule height \ht\strutbox depth \dp\strutbox##
    &&\hbox to \dimen0{\hss$##$\hss}\vrule\cr
    \noalign{\hrule}&\global\cols=2 #1\crcr
    \multispan{\the\cols}\hrulefill\cr%
   }
 }$}}
}

\begin{document}
\title[Preprojective CoHA and critical CoHA]
{On two cohomological Hall algebras}
\date{\today}

\author[Y.~Yang]{Yaping~Yang}
\address{School of Mathematics and Statistics, The University of Melbourne, 813 Swanston Street, Parkville VIC 3010, Australia}
\email{yaping.yang1@unimelb.edu.au}

\author[G.~Zhao]{Gufang~Zhao}
\address{Department of Mathematics,
University of Massachusetts, Amherst, MA, 01003, USA}
\email{gufangzhao@umass.edu}

\subjclass[2010]{
Primary 14N35;  	
Secondary 
17B37,   
14F43.
}
\keywords{Quiver variety, Hall algebra, Yangian, quiver with potentials.}

\begin{abstract}
We compare two cohomological Hall algebras (CoHA). 
The first one is the preprojective CoHA introduced in \cite{YZ15} associated to each quiver $Q$, and each algebraic oriented cohomology theory $A$. It is defined as the $A$-homology of the moduli of representations of the preprojective algebra of $Q$, generalizing the $K$-theoretic Hall algebra of commuting varieties of Schiffmann-Vasserot \cite{SV2}. The other one is the critical CoHA defined by Kontsevich-Soibelman associated to each quiver with potentials. It is defined using the equivariant cohomology with compact support with coefficients in the sheaf of vanishing cycles. In the present paper, we show that the critical CoHA, for the quiver with potential of Ginzburg, is isomorphic to the preprojective CoHA as algebras. As applications, we obtain an algebra homomorphism from the positive part of the Yangian to the critical CoHA.
 \end{abstract}
\maketitle
\tableofcontents
\section{Introduction}

In this paper we study the relation between two cohomological Hall algebras (CoHA). One arises from the study of the symmetry of the cohomology of Nakajima quiver varieties; the other  from Donaldson-Thomas theory of 3-Calabi-Yau categories. 

The former CoHA, called the preprojective CoHA, was introduced in \cite{YZ15} for each quiver $Q$ and each algebraic oriented cohomology theory $A$ in the sense of Levine-Morel \cite{LM}.  This is a CoHA associated to the 2-Calabi-Yau category of representations of the preprojective algebra of the quiver $Q=(I,H)$, where $I$ is the set of vertices and $H$ is the set of arrows. 
The preprojective CoHA, denoted by $\calP(A, Q)=\bigoplus_{v\in\bbN^I}\calP(A, Q)_v$, was used to construct an affine quantum group associated to the Kac-Moody Lie algebra $\fg_Q$ of $Q$ and $A$, which acts on the $A$-homology of quiver varieties. Nakajima-type operators are lifted as elements in the preprojective CoHA. 
In the case when $A$ is the Chow group, we recovered the action of the Yangian on quiver varieties constructed in \cite{Va00}. 
Another special case when $A$ is the $K$-theory and $Q$ is the Jordan quiver, the preprojective CoHA is the $K$-theoretic Hall algebra of commuting varieties studied by Schiffmann-Vasserot, which has been proven to be isomorphic to the elliptic Hall algebra \cite{SV2}.
The preprojective CoHA was used in \cite{YZ2} to construct affine quantum groups coming from arbitrary formal group laws, which includes the quantizations of Manin triples in \cite[\S~4]{Dr} as special cases,  and in \cite{Z15} to construct a Drinfeld realization of the elliptic quantum group and its action on the equivariant elliptic cohomology of quiver varieties which was previously unknown. The definition of this algebra is briefly recalled in \S~\ref{preproj CoHA}. 
Let $\fM(v,w)$ be the Nakajima quiver variety  with dimension vector $v,w\in\bbN^I$ (see \S\ref{subsec:PreproRepn}). By \cite[Theorem~B]{YZ15},  there is a homomorphism of $\bbN^I$-graded $\bbQ[\![t_1,t_2]\!]$-algebras 
\[
a^{\prepr}: \calP(A, Q) \to \End\Big( \bigoplus_{v\in \bbN^I} A_{G_w\times T}(\fM(v, w) ) \Big).
\]
Here $T=\Gm^2$, whose action is specified in  Assumption~\ref{Assu:WeghtsGeneral}.

The latter CoHA  was defined by Kontsevich-Soibelman in the cornerstone work \cite{KS}, associated to the category of representations of a quiver with potential, called the critical CoHA.
Recall that associated to any quiver with potential, there is a Jacobian algebra $J_W$, which is defined as the quotient of the path algebra by the ideal generated by the derivatives of the potential (\cite{DWZ}, see also \S~\ref{sec:crit coha}). The category of modules over $J_W$ is an abelian heart of a standard $t$-structure in a 3-Calabi-Yau category \cite{Keller}.  Kontsevich-Soibelman constructed a critical CoHA, whose underlying vector space is the critical cohomology of the representation space of $J_W$ \cite{KS}. Here critical cohomology means cohomology valued in a vanishing cycle.

In the present paper, we focus on the case when $A$ is the Borel-Moore homology $H^{\BM}$ for the proprojective CoHA. 
For the critical CoHA, for the same quiver $Q=(I,H)$, we consider the extended quiver with potential $(\widehat Q,W)$ studied by Ginzburg in \cite{G} .
Here the set of vertices of  $\widehat Q$ is $I$, and its set of arrows is $H\sqcup H^{\op}\sqcup C$ with $H^{\op}$ in bijection with $H$, and for each $a\in H$, the corresponding arrow in $H^{\op}$, denoted by  $a^*$, is $a$ with orientation reversed. The set 
$C$ is $\{l_i\mid i\in I\}$, with $l_i$ an edge loop at the  vertex $i \in I$.
The potential of $\widehat Q$ is
\[W:=\sum_{i\in I} l_i \cdot \sum_{a\in H}[a, a^*]. \]
We study the critical CoHA  $\calH^{\crit}(\widehat Q, W):=\bigoplus_{v\in \N^I}\calH^{\crit}(\widehat Q, W)_v$  associated to $(\widehat Q, W)$, defined in \S \ref{sec:critCoHA} in the same way as Kontsevich-Soibelman  \cite{KS}, taking into account a $T$-action on the representation space of $\widehat{Q}$. 

We construct an isomorphism of the corresponding preprojective CoHA and the critical CoHA.
Thus, the present paper can be considered as a step towards establishing a relation between instanton counting and Donaldson-Thomas invariants of Calabi-Yau 3-folds. Such a relation has been predicted in special cases on the level of generating functions in the study of geometric engineering (see \cite{Sz}). As an application of our comparison, we obtain an algebra homomorphism of the positive part of Yangian into the critical CoHA, which is an embedding for ADE type quivers.

\begin{thm}[Theorem \ref{thm:KS_prep}]\label{ThmIntr:Crit}
With the $T$-action on the representation space of $\widehat{Q}$ described in Assumption~\ref{Assu:WeghtsGeneral}, the following holds.
\begin{enumerate}
\item
For any $w\in\bbN^I$, there is a homomorphism  of $\bbN^I$-graded $\bbQ[t_1,t_2]$-algebras 
\[
a^{\crit}: \calH^{\crit}(\widehat Q, W) \to \End\Big( \bigoplus_{v\in \bbN^I} H^{\BM}_{G_w\times T}(\fM(v, w) ) \Big).
\]
\item For any $v\in\bbN^I$, let  $(-1)^{{v}\choose{2}}:=\prod_{i\in I}(-1)^{{v^i}\choose{2}}$. 
 There is an isomorphism of $\bbN^I$-graded associative  algebras $\Xi: \calP(\BM, Q)\to \calH^{\crit}(\widehat Q, W)$ whose restriction to the degree-$v$ piece is
\[
\Xi_v: \calP(\BM, Q)_v\to \calH^{\crit}(\widehat Q, W)_v, \,\ \,\
f\mapsto (-1)^{{v}\choose{2}} f.
\]
\item
In the setup of (2), 
we have
\[
a^{\crit}\big(\Xi_{v_1}(x)\big)\big((-1)^{{v_2}\choose{2}} m\big)=a^{\prepr}(x)(m)\cdot (-1)^{{v_1+v_2}\choose{2}},
\]
for any $w,v_1,v_2\in\bbN^I$, $x\in \calP_{v_1}$, and $m\in H^{\BM}_{G_w\times T}(\mathfrak{M}(v_2,w))$.
\end{enumerate}
\end{thm}

The algebra ${\calP}(\BM, Q)$ has a presentation as a shuffle algebra (see \cite[Section 3.2]{YZ15}), hence consequently so is $\calH^{\crit}(\widehat Q, W)$. Moreover, 
 for special type of quivers, ${\calP}(\BM, Q)$ is known to be related to various versions of Yangians, which are affine type quantum groups. Therefore, the above theorem also provides a relation between $\calH^{\crit}(\widehat Q, W)$ with these Yangians. 
\Omit{
In \cite[Theorem~D]{YZ15}, we proved there is an algebra homomorphism from the Yangian $ Y_\hbar^+(\fg_Q)$ to 
$\widetilde{\calP}(\BM, Q)$, where the latter is a sign twist of the preprojective CoHA ${\calP}(\BM, Q)$ via the Euler-Ringel form, which is recalled in \S\ref{sec:Yangian}. The quantization parameter $\hbar$ of the Yangian corresponds to the weight of the torus action. When the quiver $Q$ is of type ADE, it induces an isomorphism of $Y_\hbar^+(\fg_Q)$ with a spherical subalgebra of $\widetilde{\calP}(\BM, Q)$. A direct consequence of \cite[Theorem~D]{YZ15} and Theorem~\ref{ThmIntr:Crit} is the existence of an algebra homomorphism $Y_\hbar^+(\fg_Q)\to \widetilde{\calH}_{\hbar}(\widehat Q, W)$, where the latter is the same sign twist of ${\calH}(\widehat Q, W)$ via the Euler-Ringel form (see Corollary~\ref{cor:Yangian}).
} 
\Omit{
The conditions on the torus action here is essential for Theorem~\ref{ThmIntr:Crit}. In Appendix~\ref{app:sign-twist}, we investigate the two CoHAs when the $T$-action has different weights than Assumption~\ref{Assu:WeghtsGeneral}. We obtain an algebra homomorphism from one to a twisted version of the other. }

The main technical ingredient in passing from the critical cohomology to the Borel-Moore homology is the dimension reduction \cite[Theorem~A1]{D}, which is a cohomological version of the dimension reduction \cite{BBS}. In Appendix~\ref{App}, we prove that the dimension reduction is compatible with certain pull-backs and push-forwards. This allows us to give a description of the critical CoHA in terms of Borel-Moore homology, and hence prove Theorem~\ref{ThmIntr:Crit}.

\Omit{
In special cases, quiver varieties are instanton moduli spaces, and their Euler characteristic can be viewed as instanton counting. 
In other words, the Euler characteristic of certain modules of the preprojective CoHA yields invariants of instantons. 
On the other hand, in \cite{KS}, the generating series of the critical CoHA is defined as a generalization of the Donaldson-Thomas invariants of 3-folds.
Therefore, we consider Theorem \ref{ThmIntr:Crit} as a step towards establishing a relation between instanton counting and Donaldson-Thomas invariants of 3-folds. A precise problem in this direction was raised by Soibelman in \cite[\S~1.4c)]{S14}, i.e., in the case of PT-moduli of resolved conifold, find the relation between operators from the critical CoHA and the Nakajima-type operators on $\coprod_n\Hilb^n(\Aff^2)$. This special case will be discussed in Example~\ref{Ex:ResolvConifold}.
}

\subsection*{Acknowledgments} 
This paper is motivated by a discussion one of us had with Sergey Mozgovoy in the spring school on
Kac conjectures and quiver varieties at Wuppertal in 2015.  
The results of the present paper were first published on the arXiv as a part of \cite{YZ15}, and were later split therefrom due to exposition reasons. 
Shortly afterwards, Ren-Soibelman independently obtained the same results in \cite{RS}. 
Then,
Ben Davison pointed out a significant sign error in 
the earlier stated results in both \cite{YZ15} and \cite{RS},
and published a corrected version with proofs as an appendix to
\cite{RS}.
This sign error has been corrected in the present paper, and the proofs presented here are entirely different than that of Davison, although both proofs rely on \cite{D}. 
The present paper features a 2-dimensional torus action and precise conditions on the action which makes Theorem~\ref{ThmIntr:Crit} true equivariantly.

The authors thank Ben Davison for pointing out the aforementioned error in an earlier version of Theorem~\ref{ThmIntr:Crit}, and for drawing our attention to the appendix of \cite{RS}.
Most of the work was done when Y.Y. was hosted by MPIM and  G.Z. was supported by Fondation Sciences Math\'ematiques de Paris and CNRS.

\section{The critical cohomological Hall algebra} \label{sec:critCoHA}
Let $\Gamma=(\Gamma_0, \Gamma_1)$ be a quiver, where $\Gamma_0$ is the set of vertices, and $\Gamma_1$ the set of arrows. We denote the path algebra of $\Gamma$  by $\C \Gamma$. 
Let $W$ be a potential of $\Gamma$, that is, $W=\sum_{u} c_u u$ with $c_{u} \in \C$, and $u$'s are cycles in $\Gamma$. 

Given a cycle $u=a_1\dots a_n$ and an arrow $a\in \Gamma_1$. The cyclic derivative is defined to be
\[
\frac{\partial u}{\partial a}=\sum_{i: a_i=a} a_{i+1}\dots a_n a_1\dots a_{i-1} \in \C \Gamma.
\] as an element of $\C \Gamma$. We extend the cyclic derivative to the potential by linearity.

For any dimension vector $v=(v^i)_{i\in \Gamma_0}\in \bbN^{\Gamma_0}$, the representation space of $\Gamma$ with dimension vector $v$ is denoted by $\bold{M}_{\Gamma, v}$. That is, let $V=\{V^i\}_{i\in \Gamma_0}$ be a $|\Gamma_0|$-tuple of vector spaces so that $\dim(V^i)=v^i$.  Then, 
\[
\bold{M}_{\Gamma, v}:=\bigoplus_{h\in \Gamma_1}\Hom(V^{\out(h)},V^{\inc(h)}).
\]
The group $G_{v}:=\prod_{i\in I}\GL_{v^i}$ acts on the representation space $\bold{M}_{\Gamma, v}$ via conjugation.

We recall the critical CoHA defined by Kontsevich-Soibelman, and give an equivalent description for a special type of potential interesting to us.

\subsection{Critical CoHA via critical cohomology}\label{subsec:app_KS_crit_coh}
For a quiver with potential $(\Gamma, W)$ and dimension vector $v\in \bbN^I$,  denote by
$\tr(W)_v$ the trace function on $\bold{M}_{\Gamma, v}$. Let $\Crit(\tr W_v)$ be the critical locus of $\tr W_v$. Let $\varphi_{\tr W_{v}}$ be the  vanishing cycle complex on $\bold{M}_{\Gamma, v}$ with support on $\Crit(\tr W_v)$.
We have an isomorphism 
\[H_{c, G_v}^*(\bold{M}_{\Gamma, v}, \varphi_{\tr W_v})\cong H_{c, G_v}^*(\Crit(\tr W_v), \varphi_{\tr W_v}).\]

For $v_1,v_2\in \bbN^I$ such that $v=v_1+v_2$, let $V_1\subset V $ be a $|\Gamma_0|$-tuple of subspaces of $V$ with dimension vector $v_1$.

Define $\bold{M}_{\Gamma, v_1, v_2}:=\{x\in \bold{M}_{\Gamma, v}\mid x(V_1)\subset V_1\}$.  We write $G:=G_{v}$ for short. Let $P\subset G_v$ be the parabolic subgroup preserving the subspace $V_1$ and $L:=G_{v_1}\times G_{v_2}$ be the Levi subgroup of  $P$. We have the following correspondence of $L$-varieties. 
\begin{equation}\label{basic corresp}
\xymatrix
{\bold{M}_{\Gamma, v_1}\times \bold{M}_{\Gamma, v_2}&\bold{M}_{\Gamma, v_1, v_2} \ar[l]_(0.4){p}\ar[r]^(0.4){\eta} &\bold{M}_{\Gamma, v_1+v_2},
}\end{equation}
where $p$ is the natural projection and $\eta$ is the embedding. 
The trace functions $\tr W_{v_i}$ on $\bold{M}_{\Gamma, v_i}$ induce a function $\tr W_{v_1}\boxplus \tr W_{v_2}$
on the product $\bold{M}_{\Gamma, v_1}\times \bold{M}_{\Gamma, v_2}$.
We define $\tr(W)_{v_1, v_2}$ on $\bold{M}_{\Gamma, v_1, v_2}$ to be
\[
\tr(W)_{v_1, v_2}:=p^*(\tr W_{v_1}\boxplus \tr W_{v_2})=\eta^*(\tr W_{v_1+v_2}).
\]
Note that we have
$
p^{-1} (\Crit(\tr W_{v_1})\times \Crit(\tr W_{v_2}))
\supsetneqq
\eta^{-1}(\Crit(\tr W_{v_1+v_2})).
$

For simplicity, we assume that $\Crit(\tr W_v)\subseteq (\tr W_v)^{-1}(0)$ so that we have the Thom-Sebastiani isomorphism as below \cite{M01}.

The Hall multiplication $m^{\crit}$ of the critical CoHA is the composition of the following \cite{KS}.
\begin{enumerate}
\item The Thom-Sebastiani isomorphism
\begin{align*}
&H_{c, G_{v_1}}^*(\bold{M}_{\Gamma, v_1}, \varphi_{\tr W_{v_1}})^\vee \otimes 
H_{c, G_{v_2}}^*(\bold{M}_{\Gamma, v_2}, \varphi_{\tr W_{v_2}})^\vee
\cong 
H_{c, L}^*(\bold{M}_{\Gamma, v_1}\times \bold{M}_{\Gamma, v_2}, \varphi_{\tr W_{v_1} \boxplus \tr W_{v_2}} )^\vee.
\end{align*}
\item 
Using the fact that $\bold{M}_{\Gamma, v_1, v_2}$
is an affine bundle over $\bold{M}_{\Gamma, v_1}\times \bold{M}_{\Gamma, v_2}$, and
$\tr W_{v_1, v_2}$ is the pullback 
of $\tr W_{v_1} \boxplus \tr W_{v_2}$, we have
\begin{align*}
p^*:H_{c, L}^*(\bold{M}_{\Gamma, v_1}\times \bold{M}_{\Gamma, v_2},
 \varphi_{\tr W_{v_1} \boxplus \tr W_{v_2}})^\vee
\cong
&H_{c, L}^*(\bold{M}_{\Gamma, v_1, v_2}, 
\varphi_{\tr W_{v_1, v_2} } )^{\vee}
\end{align*}
\item
Using the fact $\tr W_{v_1, v_2}$ is the restriction of $\tr W_{v}$ to
$\bold{M}_{\Gamma, v_1, v_2}$. We have
\begin{align*}
\eta_*:H_{c, L}^*(\bold{M}_{\Gamma, v_1, v_2}, 
\varphi_{\tr W_{v_1, v_2} } )^{\vee}\to &
H_{c, L}^*(\bold{M}_{\Gamma, v}, 
\varphi_{\tr W_{v} } )^{\vee}.
\end{align*}
\item 
Pushforward along 
$G\times_{P} \bold{M}_{\Gamma, v}\to \bold{M}_{\Gamma, v}, (g, m)\mapsto gmg^{-1}$,
 we get
\[
H_{c,L}(\bold{M}_{\Gamma, v}, 
\varphi_{\tr W_{v} } )^{\vee}\cong H_{c,P}(\bold{M}_{\Gamma, v}, 
\varphi_{\tr W_{v} } )^{\vee}\cong H_{c, G}^*(G\times_{P} \bold{M}_{\Gamma, v}, 
\varphi_{\tr W_{v} } )^{\vee}\to 
H_{c, G}^*(\bold{M}_{\Gamma, v}, 
\varphi_{\tr W_{v} } )^{\vee}.
\]
\end{enumerate}

\begin{definition}\label{def:crit CoHA in KS}
The critical CoHA is $\calH^{\crit}(\Gamma, W):=\bigoplus_{v\in \N^I}H_{c, G_v}^*(\bold{M}_{\Gamma, v}, \varphi_{\tr W_v})^\vee$, endowed with the Hall multiplication $m^{\crit}$ described above.
\end{definition}

\begin{remark}
\begin{enumerate}
\item 
Let $J_{\Gamma, W}$ be the Jacobian algebra of the quiver with potential $(\Gamma, W)$. (See \cite{DWZ} for details.)
Representations of the Jacobian algebra $\Rep(J_{\Gamma, W}, v)$ as a Zariski closed subvariety
of $\bold{M}_{\Gamma, v}$ is the same as $\Crit(\tr W_v)$. Thus, we have
\[
H_{c, G_v}^*(\bold{M}_{\Gamma, v}, \varphi_{\tr W_v})\cong 
H_{c, G_v}^*(\Rep(J_{\Gamma, W}, v), \varphi_{\tr W_v}).
\]
\item 
The definition of critical CoHA in \cite{KS} is more general. 
The critical cohomology of more general subvarieties $M_{v}^{\sp1} \subset \Crit(\tr W_v)$ was considered. 
In our setup, we only take the maximal choice $M_{v}^{\sp1} =\Crit(\tr W_v)$, and we assume the quiver with potential admits a cut satisfying Assumption \ref{Assu}.
\Omit{
\item The assumption that $\Crit(\tr W_v)$ is contained in the zero-locus of $\tr W_v$ is always true for the cases interesting to us, i.e., under Assumption~\ref{Assu} by \eqref{equ:trW}. Readers interested in more general cases can take suitable choice of $M_{v}^{\sp1}$ to avoid this issue.
}
\item In the definition given in \cite{KS}, cohomological degree was taken into consideration.
\end{enumerate}
\end{remark}

\subsection{Quiver with potential and cut}\label{sec:crit coha}

A cut $C$ of $(\Gamma, W)$ is a subset $C\subset \Gamma_1$ such that $W$ is homogeneous of degree 1 with respect to the grading defined on arrows by

\begin{displaymath} 
\deg a= 
\left\{
     \begin{array}{lr}
       1 & : a \in C, \\
       0 & : a \notin C.
     \end{array}
   \right.
   \end{displaymath}
In this section, we assume the quiver with potential $(\Gamma, W)$ admits a cut $C$. Furthermore, we assume the following.
\begin{assumption}
\label{Assu}
The cut $C$ consists of only edge-loops. 
\end{assumption}
In particular under Assumption \ref{Assu}, each term in $W$ has at least one edge-loop.
Moreover, we have the isomorphism $
\bold{M}_{C, v}\cong \bigoplus_{i\in \Gamma_0} (\fg\fl_{v^i})^{n_i}$, where $n_i$ is the number of loops in $C$ at vertex $i\in \Gamma_0$. As before, we have $\bold{M}_{C, v_1, v_2}=\{x\in \bold{M}_{C, v_1+v_2}\mid x(V_1)\subset V_1\}$. 
Note that if $C$ consists of exactly one edge-loop for each vertex, we identify $\bold{M}_{C, v}$ with $\fg_v=\bigoplus_{i\in \Gamma_0} \fg\fl_{v^i}$, the Lie algebra of $G_v$, and $\bold{M}_{C, v_1, v_2}$ with the parabolic subalgebra $\fp_{v_1, v_2}$, which is the Lie algebra of $P_{v_1, v_2}$.

Let $\Gamma\backslash C$ be the new quiver obtained from $\Gamma$ by removing the cut $C$. 
We have the decomposition $\bold{M}_{\Gamma, v}\cong 
\bold{M}_{\Gamma\backslash C, v}\oplus \bold{M}_{C, v}$ for $v\in \mathbb{N}^{\Gamma_0}$. Write $v=v_1+v_2$.  By \eqref{basic corresp}, there are the two correspondences $\bold{M}_{C, v_1}\times \bold{M}_{C, v_2}\leftarrow \bold{M}_{C, v_1, v_2}\rightarrow \bold{M}_{C, v}$, and 
$\bold{M}_{\Gamma \backslash C, v_1}\times \bold{M}_{\Gamma \backslash C, v_2} \leftarrow 
\bold{M}_{\Gamma \backslash C, v_1, v_2} \rightarrow \bold{M}_{\Gamma \backslash C, v}$. It induces the following diagram with the square being Cartesian
\begin{equation}
\label{diag:j23}
\xymatrix@R=1.5em  {
\bold{M}_{\Gamma, v_1}\times \bold{M}_{\Gamma, v_2} 
& 
(\bold{M}_{\Gamma \backslash C, v_1}\times \bold{M}_{\Gamma \backslash C, v_2})\oplus
\bold{M}_{C, v_1, v_2}
\ar[l]_(0.6){p_1}\ar[r]^{i_1}
\ar@/^/@{.>}[d]^{j_3}
& 
(\bold{M}_{\Gamma \backslash C, v_1}\times \bold{M}_{\Gamma \backslash C, v_2}) \oplus 
\bold{M}_{C, v}
\ar@/^/@{.>}[d]^{j_2\times\id_{\bold{M}_{C, v}}}
\\
&\bold{M}_{\Gamma \backslash C, v_1, v_2}
 \oplus  \bold{M}_{C, v_1, v_2}
\ar[u]^{p_3}
\ar[ul]^{p_{\Gamma}} \ar[rd]_{{\eta_\Gamma}}\ar@{^{(}->}[r]^{i_3} &
\bold{M}_{\Gamma \backslash C, v_1, v_2}
\oplus  \bold{M}_{C, v}
 \ar[d]_{{i_2}}\ar[u]^{{p_2}} \\
&&  \bold{M}_{\Gamma \backslash C, v}\oplus \bold{M}_{C, v}.
}\end{equation}
The embeddings $j_2\times \id_{\bold{M}_{C, v}}$ and $j_3:=j_2\times \id_{\bold{M}_{C, v_1, v_2}}$ 
come from the inclusion $j_2: \bold{M}_{\Gamma \backslash C, v_1}\times \bold{M}_{\Gamma \backslash C, v_2}\subset \bold{M}_{\Gamma \backslash C, v_1, v_2}$.

For each $a\in C$, the derivative 
$\frac{\partial W}{\partial a}$ is a linear combination of cycles of $\Gamma\backslash C$, as $W$ is homogenous of degree 1. For any  $x\in \bold{M}_{\Gamma\backslash C, v}$, denote by $\frac{\partial W}{\partial a}(x)$ the composition of the linear maps $x=(x_h)_{h\in \Gamma_1\backslash C}$ along $\frac{\partial W}{\partial a}$. Then, $(\frac{\partial W}{\partial a}(x))_{a\in C}$ is an element in  $\bigoplus_{i\in \Gamma_0} (\fg\fl_{v^i})^{n_i}\cong \bold{M}_{C, v}$. 

Consider the quotient of the path algebra $\C (\Gamma \backslash C)$ by the relations 
$\{{\partial W}/{\partial a} \mid a\in C \}$. The representation variety of this quotient algebra is denoted by
\[
\bold{J}_{\Gamma\backslash C, v}:=\{ x\in \bold{M}_{\Gamma \backslash C, v} \mid {\partial W}/{\partial a}(x)=0,  \forall a\in C \}.\] 
We have the inclusion 
$\bold{J}_{\Gamma\backslash C, v}\oplus \bold{M}_{C, v}\subset  \bold{M}_{\Gamma\backslash C, v}\oplus  \bold{M}_{C, v}=\bold{M}_{\Gamma, v}$.
The two maps $p_1, i_1$ in diagram \eqref{diag:j23} induce the following two maps
 \[\xymatrix
 {
 \bold{J}_{\Gamma\backslash C, v_1}\times \bold{J}_{\Gamma\backslash C, v_2}\times  \bold{M}_{C, v_1}\times \bold{M}_{C, v_2}&
 \bold{J}_{\Gamma\backslash C, v_1}\times \bold{J}_{\Gamma\backslash C, v_2}\times \bold{M}_{C, v_1, v_2}
 \ar[l]_(0.45){ \overline{p_1}}\ar[r]^{\overline{i_1}}
 &
 \bold{J}_{\Gamma\backslash C, v_1}\times \bold{J}_{\Gamma\backslash C, v_2}\times  \bold{M}_{C, v}
 }\]
Define
\begin{align*}
\bold{J}_{\Gamma\backslash C, v_1, v_2}:=
&\bold{M}_{\Gamma\backslash C, v_1, v_2} \cap  \bold{J}_{\Gamma\backslash C, v_1+v_2}
= 
\{x \in \bold{M}_{\Gamma\backslash C, v_1, v_2} \mid 
{\partial W}/{\partial a}(x)=0,  \forall a\in C\}.
\end{align*}
In diagram \eqref{diag:j23}, replacing  $\bold{M}_{\Gamma\backslash C, v}$  by $\bold{J}_{\Gamma\backslash C, v}$ and similarly for the other dimension vectors, we get a similar diagram with the maps denoted by the original maps with an overline. 
But  $\bold{J}_{\Gamma\backslash C, v_1, v_2}=\eta_{\Gamma\backslash C}^{-1}( \bold{J}_{\Gamma\backslash C, v_1+v_2})\subsetneqq p_{\Gamma\backslash C}^{-1} (\bold{J}_{\Gamma\backslash C, v_1}\times \bold{J}_{\Gamma\backslash C, v_2} )$.

Let $\pr: \bold{M}_{C, v_1,v_2}\to \bold{M}_{C, v_1}\times \bold{M}_{C, v_2}$ be the natural projection. 
Consider the following subvariety $\bold{Y}\subset \bold{M}_{C, v_1,v_2}\times \bold{M}_{\Gamma\backslash C, v_1}\times \bold{M}_{\Gamma\backslash C, v_2}$. 
\[
\bold{Y}:=\{(l, x)\mid  l\in \bold{M}_{C, v_1,v_2}, x\in \bold{M}_{\Gamma\backslash C, v_1}\times \bold{M}_{\Gamma\backslash C, v_2},  \,\ \text{such that $({\partial W}/{\partial a})_{a\in C}(x)=\pr(l)$}\}.
\]
We have the following diagram.
\begin{equation}\label{corresp with fiber}
\xymatrix@R=1.5em@C=2.5em{
\bold{M}_{\Gamma\backslash C, v_1}\times \bold{M}_{\Gamma\backslash C, v_2}
\times \bold{M}_{C, v}
\ar@{^{(}->}[r]^(0.7){\iota\times\id_{\bold{M}_{C, v}}}&\bold{Y}\times \bold{M}_{C, v}&
\bold{M}_{\Gamma\backslash C, v_1, v_2}\times \bold{M}_{C, v}  
\ar[l]_(0.6){\omega\times\id_{\bold{M}_{C, v}}} \ar[r]&
\bold{M}_{\Gamma \backslash C, v} \times \bold{M}_{C, v}
\\
\bold{J}_{\Gamma\backslash C, v_1}\times \bold{J}_{\Gamma\backslash C, v_2}\times \bold{M}_{C, v}
\ar@{^{(}->}[u]
&&
\bold{J}_{\Gamma\backslash C, v_1, v_2} \times \bold{M}_{C, v} 
\ar[ll]_{\overline{\omega}\times\id_{\bold{M}_{C, v}}}\ar[r]^{\overline{i_2}}\ar@{^{(}->}[u]&
\bold{J}_{\Gamma\backslash C, v} \times \bold{M}_{C, v} \ar@{^{(}->}[u]
}
\end{equation}
The maps in diagram \eqref{corresp with fiber} are as follows. 
Any map in \eqref{corresp with fiber} restricting on $\bold{M}_{C, v}$ is the identity map. 
The map $\omega$ is induced from the natural map 
$\bold{M}_{\Gamma\backslash C, v_1, v_2} \to \bold{Y}, x\mapsto (({\partial W}/{\partial a})_{a\in C}(x) , \pr(x))$, where $\pr:\bold{M}_{\Gamma\backslash C, v_1, v_2} \to\bold{M}_{\Gamma\backslash C, v_1}\times \bold{M}_{\Gamma\backslash C, v_2}$ is the natural projection. The map $\overline{\omega}$ is the restriction of $\omega$.
The map $\iota: \bold{M}_{\Gamma\backslash C, v_1}\times \bold{M}_{\Gamma\backslash C, v_2}
\inj \bold{Y}$ is given by $x\mapsto (({\partial W}/{\partial a})_{a\in C}(x), x)$.

\begin{lemma}\label{pullback}
There is a natural isomorphism $$(\bold{J}_{\Gamma\backslash C, v_1}\times \bold{J}_{\Gamma\backslash C, v_2})\times_{\bold{Y}} \bold{M}_{\Gamma\backslash C, v_1, v_2}\cong \bold{J}_{\Gamma\backslash C, v_1, v_2}.$$
\end{lemma}
\begin{proof}
Indeed,
\begin{align*}
&(\bold{J}_{\Gamma\backslash C, v_1}\times \bold{J}_{\Gamma\backslash C, v_2})\times_{\bold{Y}} \bold{M}_{\Gamma\backslash C, v_1, v_2}\\
=&
\{
(x_1, x_2)\in \bold{J}_{\Gamma\backslash C, v_1}\times \bold{J}_{\Gamma\backslash C, v_2}, 
x \in \bold{M}_{\Gamma\backslash C, v_1, v_2} \mid 
\pr(x)=(x_1, x_2),  ({\partial W}/{\partial a})_{a\in C}(x)=0\}\\
=&\{ x \in \bold{M}_{\Gamma\backslash C, v_1, v_2} \mid 
({\partial W}/{\partial a})_{a\in C}(x)=0\}
=\bold{J}_{\Gamma\backslash C, v_1, v_2}.
\end{align*}
This completes the proof.
\end{proof}
\subsection{Another description of the critical CoHA}
\label{sec:aux CoHA in sec1}
We follow the convention in \cite{D},  let $D^b(X)$ be the derived category of constructible sheaves of 
$\bbQ$-vector spaces on a variety $X$, and $\mathbb{D}$ be the Verdier duality functor for $D^b(X)$. 
Denote by $H_{c}^*(X)^\vee$ the Verdier dual of the compactly supported cohomology of $X$. Write the structure map $X\to \pt$  as $p_X$. The Borel-Moore homology of $X$ is
\[
H^{\BM}(X) \cong H_{c}^*(X)^\vee\cong \mathbb{D}_{\pt}(p_{X!} \Q_{X})=p_{X*} 
\mathbb{D}  \Q_{X}.
\]
If $X$ carries a $G$-action, we denote by $H_{c, G}^*(X)^\vee$ the corresponding equivariant cohomology of $X$.

The following  equivalent description of the critical CoHA using Borel-Moore homology will be proved in \S~\ref{sec:proof1}. Let $e(\iota)$ be the equivariant Euler class of the embedding $\iota$ in \eqref{corresp with fiber}. 
\begin{theorem}\label{prop: crit CoHA with the one in KS}
Assume the quiver with potential $(\Gamma, W)$ admits a cut $C$
satisfying Assumption \ref{Assu}. Then there is an isomorphism of graded vector spaces $\calH^{\crit}(\Gamma, W)\cong \bigoplus_{v\in \bbN^{\Gamma_0}} H^{\BM}_{G_v}(\bold{J}_{\Gamma\backslash C, v} \times \bold{M}_{C, v}, \bbQ)$, under which the multiplication $m^{\crit}$ on $\calH^{\crit}(\Gamma, W)$ is equal to
$\overline{i_{2}}_*  \circ \frac{1}{e(\iota)}\omega_{\overline{\omega}}^{\sharp} \circ \overline{i_{1}}_* \circ \overline{p_1}^{*}$. 
\end{theorem}
The maps in the composition are the following. 
\begin{enumerate}
\item The K\"unneth morphism
$
H^{\BM}_{G_{v_1}}(\bold{J}_{\Gamma\backslash C, v_1} \times \bold{M}_{C, v_1})
\otimes H^{\BM}_{G_{v_2}}(\bold{J}_{\Gamma\backslash C, v_2} \times \bold{M}_{C, v_2})
 \to 
H^{\BM}_{L}(\bold{J}_{\Gamma\backslash C, v_1}\times \bold{J}_{\Gamma\backslash C, v_2} \times \bold{M}_{C, v_1}\times  \bold{M}_{C, v_2}). 
$
\item 
$\overline{i_{1}}_* \circ \overline{p_1}^{*}: 
H^{\BM}_{L}(\bold{J}_{\Gamma\backslash C, v_1}\times \bold{J}_{\Gamma\backslash C, v_2}  \times \bold{M}_{C, v_1}\times  \bold{M}_{C, v_2})\to
H^{\BM}_{L}(\bold{J}_{\Gamma\backslash C, v_1} \times \bold{J}_{\Gamma\backslash C, v_2}  \times \bold{M}_{C, v}). 
$
\item 
Composing the refined Gysin pullback\footnote{We follow the convention of \cite[\S~6.6.2]{LM} for refined Gysin pullback and use subindex to remember the domain and target. However, we adapt $\sharp$ instead of $!$ to avoid confusion with $!$-pullback.} of $\omega\times\id_{\bold{M}_{C,v}}$ along $\overline{\omega}\times\id_{\bold{M}_{C,v}}$, which for simplicity  is denoted by $\omega_{\overline{\omega}}^{\sharp}$, with $\frac{1}{e(\iota)}$, we have the following map
\[
\frac{1}{e(\iota)}\omega_{\overline{\omega}}^{\sharp}: H^{\BM}_{L}( 
\bold{J}_{\Gamma\backslash C, v_1}\times \bold{J}_{\Gamma\backslash C, v_2} \times \bold{M}_{C, v})
\to H^{\BM}_{L}( \bold{J}_{\Gamma\backslash C, v_1, v_2} \times \bold{M}_{C, v} )[\frac{1}{e(\iota)}]. 
\] 
\item
The pushforward $
\overline{i_{2}}_* : 
H^{\BM}_{L}( \bold{J}_{\Gamma\backslash C, v_1, v_2} \times \bold{M}_{C, v} )\to 
H^{\BM}_{L}( \bold{J}_{\Gamma\backslash C, v} \times \bold{M}_{C, v} )
$. 
\item
Pushforward along 
$G\times_{P} (\bold{J}_{\Gamma\backslash C, v} \times \bold{M}_{C, v})\to \bold{J}_{\Gamma\backslash C, v} \times \bold{M}_{C, v}, (g, m)\mapsto gmg^{-1}$,
 we get
$H^{\BM}_{P}( \bold{J}_{\Gamma\backslash C, v} \times \bold{M}_{C, v} )
\cong 
H^{\BM}_{G}( G\times_P(\bold{J}_{\Gamma\backslash C, v} \times \bold{M}_{C, v} ))
\to 
H^{\BM}_{G}( \bold{J}_{\Gamma\backslash C, v} \times \bold{M}_{C, v} )
$.
\end{enumerate}

This composition $\overline{i_{2}}_*  \circ \frac{1}{e(\iota)}\omega_{\overline{\omega}}^{\sharp} \circ \overline{i_{1}}_* \circ \overline{p_1}^{*}$ {\it a priori} is only defined after inverting $e(\iota)$. However, it follows from Theorem~\ref{prop: crit CoHA with the one in KS} that it is  well-defined before localization. 

\begin{remark}\label{rem:dependent}
The multiplication $\overline{i_{2}}_*  \circ \frac{1}{e(\iota)}\omega_{\overline{\omega}}^{\sharp} \circ \overline{i_{1}}_* \circ \overline{p_1}^{*}$ does not depend on the choices of $\bold{Y}$ in correspondence \eqref{corresp with fiber}, as long as Lemma~\ref{pullback} holds. Indeed, assume we have the following diagram 
\[
\xymatrix@R=.7em {
\bold{Y'} & \bold{M}_{\Gamma\backslash C, v_1, v_2} \ar[l]_{\omega'}\\
\bold{Y} \ar[u]&\bold{M}_{\Gamma\backslash C, v_1, v_2}\ar[l]_{\omega} \ar@{=}[u]\\
\bold{J}_{\Gamma\backslash C, v_1}\times \bold{J}_{\Gamma\backslash C, v_2}
 \ar[u]& \bold{J}_{\Gamma\backslash C, v_1, v_2} \ar[l]_(0.4){\overline{\omega}} \ar[u]\\
}
\] with all squares Cartesian and $\bold{Y}\to \bold{Y'}$ an closed embedding of manifolds. Then, by the excess intersection formula (see \cite[Theorem 6.6.9]{LM}), we have
$\omega'^{\sharp}_{\overline{\omega}} =
\omega_{\overline{\omega}}^{\sharp}e(E)$, where $E$ is the excess normal bundle of $\bold{Y}\to \bold{Y'}$. This implies that 
\[
\frac{1}{e(\iota)}\omega_{\overline{\omega}}^{\sharp}=\frac{1}{e(\iota')} \omega'^{\sharp}_{\overline{\omega}}.
\]
\end{remark}
\subsection{An auxiliary CoHA}\label{subsec:AuxCoHA}
By Theorem \ref{prop: crit CoHA with the one in KS}, we can alternatively define the critical CoHA as the vector space $\calH^c(\Gamma,W,C)=\bigoplus_{v\in \bbN^{\Gamma_0}} H^{\BM}_{G_v}(\bold{J}_{\Gamma\backslash C, v} \times \bold{M}_{C, v}, \bbQ)$ endowed with the multiplication $\overline{i_{2}}_*  \circ \frac{1}{e(\iota)}\omega_{\overline{\omega}}^{\sharp} \circ \overline{i_{1}}_* \circ \overline{p_1}^{*}$.  
This description has the advantage of being defined for any algebraic oriented cohomology theory in the sense of Levine-Morel \cite{LM}, without resorting to a theory of "exponential motives". That is, for any oriented cohomology theory $A$, we can define
$\calH^c(A,\Gamma,W,C)$ to be the graded vector space $\bigoplus_{v\in \bbN^{\Gamma_0}} A_{G_v}(\bold{J}_{\Gamma\backslash C, v} \times \bold{M}_{C, v}, \bbQ)$ endowed with the multiplication $m^{\crit}=\overline{i_{2}}_*  \circ \frac{1}{e(\iota)}\omega_{\overline{\omega}}^{\sharp} \circ \overline{i_{1}}_* \circ \overline{p_1}^{*}$.
 
For example, $A$ can be taken to be the intersection theory $\CH$ (that is, the Chow group, see \cite{Ful}). There is a cycle map $\cl: \CH_{*}(X)\to H_{2*}^{\BM}(X)$. It induces an algebra homomorphism
 $\calH(\CH, \Gamma, W, C) \to \calH(\BM, \Gamma, W, C)$. 

As an intermediate object in comparing the preprojective CoHA in \S\ref{preproj CoHA} with the critical CoHA of Kontsevich-Soibelman, we further introduce an {\it auxiliary cohomological Hall algebra}  associated to the data 
$(\Gamma, W, C)$ and $A$,  as
 \[
\calH^{\aux}(\Gamma, W, C):=\bigoplus_{v\in \N^{\Gamma_0}} \calH^{\aux}(\Gamma, W, C)_{v}=
\bigoplus_{v\in \N^{\Gamma_0} } A_{G_v\times T} (\bold{J}_{\Gamma\backslash C, v} \times \bold{M}_{C, v}, \bbQ),
 \]  with multiplication defined by the composition 
 \[
 m^{\aux}_{v_1,v_2}:=\overline{i_{2}}_*  \circ \omega_{\overline{\omega}}^{\sharp} \circ \overline{i_{1}}_* \circ \overline{p_1}^{*}.
 \]
Unlike $m^{\crit}$,  $m^{\aux}_{v_1,v_2}$ does depend on the  specific choices of $\bold{Y}$ in the correspondence \eqref{corresp with fiber}.

\begin{prop}\label{prop:critial COHA}
	The maps $\{m^{\aux}_{v_1,v_2}\}_{v_1, v_2\in \bbN^I}$  define an associative $\bbN^I$-graded algebra structure on $\calH^{\aux}(\Gamma, W, C)$.
\end{prop}
\begin{proof}
	This follows from a similar proof as that of \cite[Theorem~4.1]{YZ15}.
\end{proof}

\section{Proof of Theorem~\ref{prop: crit CoHA with the one in KS}}\label{sec:proof1}
In this section, we prove Theorem \ref{prop: crit CoHA with the one in KS} and hence work under the assumption therein. 
\subsection{Proof of Theorem~\ref{prop: crit CoHA with the one in KS}}
The proof of Theorem~\ref{prop: crit CoHA with the one in KS} amounts to showing that the dimension reduction in Theorem~\ref{thm:DavA1} is compatible with some pullbacks and pushforwards. We present some general facts in Appendix~\ref{App} for the convenience of the readers.

In the setup of Theorem~\ref{thm:DavA1}, we take $X= \bold{M}_{\Gamma \backslash C, v}, Y= \bold{M}_{\Gamma, v}, \mathbb{A}^n= \bold{M}_{C, v} $, and $f=\tr W_v$. Then, we have the subvariety 
$Z=\{x\in  \bold{M}_{\Gamma \backslash C, v}\mid \tr(W_v) (x, l)=0, \forall l\in \bold{M}_{C, v} \}$. 
\begin{lemma}\label{lem:about Z}
The subvariety $Z$ of $\bold{M}_{\Gamma \backslash C, v}$ can be identified with $\bold{J}_{\Gamma\backslash C, v}$. 
\end{lemma}
\begin{proof}
By definition, the potential $W$ is homogenous of degree $1$. 
This implies 
\begin{equation}\label{equ:trW}
\tr W=\tr \sum_{a\in C} (\partial W/\partial a) a.
\end{equation}
For any $x\in \bold{J}_{\Gamma \backslash C, v}$, we have $\partial W/\partial a (x)=0$, for all $a\in C$. Clearly, the equality \eqref{equ:trW} implies $\tr(W_v) (x, l)=\tr \sum_{a\in C} (\partial W/\partial a (x) \cdot l_a)=0$, for all $l=(l_a)_{a\in C} \in \bold{M}_{C, v}$. This shows $x\in Z$. 

On the other hand, for any $x\in Z$, we have $\tr(W_v) (x, l)=0$, for all $l=(l_a)_{a\in C} \in \bold{M}_{C, v}$. Then, by the equality \eqref{equ:trW} we have $\tr((\partial W/\partial a)(x) \cdot  l_a)=0$, for any matrix $l_a$. This shows the vanishing $(\partial W/\partial a)(x)=0$, for all $a\in C$. Therefore, $x\in \bold{J}_{\Gamma \backslash C, v}$.
\end{proof}

Lemma \ref{lem:about Z}, together with Theorem~\ref{thm:DavA1} and Proposition \ref{Ts and Ku}, yields the following.
\begin{corollary}
\label{cor:vanish cyc and usual}
There is a canonical isomorphism as vector spaces,
\begin{align*}
H_{c, G_v}^*(\bold{M}_{\Gamma, v}, \varphi_{\tr W_v})^{\vee}&
\cong H_{c, G_v}^*(\bold{J}_{\Gamma\backslash C, v}\times \bold{M}_{C, v}
, \Q)^{\vee},  \,\ \text{for $v\in \bbN^{\Gamma_0}$},
\end{align*} which is compatible with the 
 Thom-Sebastiani isomorphism and the K\"unneth isomorphism.
\end{corollary}

The following lemma will be proved in \S~\ref{subsec:proof of lemma}.

\begin{lemma}\label{lem:app_euler_phi_van}
The isomorphism in Theorem~\ref{thm:DavA1} intertwines 
\[e(\iota) (j_{2}\times\id_{\bold{M}_{C,v}})_*:H_{c, L}^*(\bold{M}_{\Gamma\backslash C, v_1}\times \bold{M}_{\Gamma\backslash C, v_2} \times \bold{M}_{C, v} , 
\varphi_{\tr W_{v_1+ v_2} } )^{\vee}\to H_{c, L}^*( \bold{M}_{\Gamma\backslash C, v_1, v_2}\times \bold{M}_{C, v} , 
\varphi_{\tr W_{v_1+v_2} } )^{\vee}\] 
and \[
\xymatrix@C=5em @R=1em{
 H^{\BM}_{L}(\bold{J}_{\Gamma\backslash C, v_1}\times \bold{J}_{\Gamma\backslash C, v_2}\times \bold{M}_{C, v})\ar[r]\ar[d]_{\cong} &H^{\BM}_{L}(\bold{J}_{\Gamma\backslash C, v_1, v_2}\times \bold{M}_{C, v})\ar[d]_{\cong}\\
H^{\BM}_{L}(\bold{J}_{\Gamma\backslash C, v_1}\times \bold{J}_{\Gamma\backslash C, v_2})[-\dim \bold M_{C,v}] \ar[r]|-{\omega^{\sharp}_{\overline{\omega}} e(j_{2})} &H^{\BM}_{L}(\bold{J}_{\Gamma\backslash C, v_1, v_2})[-\dim \bold M_{C,v}] .
}\]
\end{lemma}

\begin{proof}[Proof of Theorem~\ref{prop: crit CoHA with the one in KS}]
With notations in \eqref{diag:j23} and \eqref{eqn:iota}, the composition of \S\ref{subsec:app_KS_crit_coh}(2) and (3) in the multiplication on $\calH^{\crit}(\Gamma, W)$ of Kontsevich-Soibelman is equivalent to 
\[
i_{2*}\circ i_{3*}\circ p_{3}^*\circ p_1^*:H^*_{c,L}
(\bold{M}_{\Gamma, v_1}
\times \bold{M}_{\Gamma, v_2},\varphi_{\tr W_{v_1}\boxplus W_{v_2}})^\vee\to 
H^*_{c,L}(\bold{M}_{\Gamma, v_1+v_2},\varphi_{\tr W_{v_1+v_2}})^\vee.\]
In the following 4 steps, we show that $i_{2*}\circ i_{3*}\circ p_{3}^*\circ p_1^*$ is the same as $\overline{i_2}_* \circ \frac{1}{e(\iota)}\omega_{\overline{\omega}}^{\sharp} \circ \overline{i_1}_* \circ \overline{p_1}^*$ under the isomorphism of Corollary~\ref{cor:vanish cyc and usual}.

{\bf Step 1. }We have the following commutative diagram
\[
\xymatrix@R=1.5em@C=1em {
H^*_{c,L}
(\bold{M}_{\Gamma, v_1}
\times \bold{M}_{\Gamma, v_2},\varphi_{\tr W_{v_1}\boxplus W_{v_2}})^\vee\ar[r]^(0.45){{p_1}^*}\ar[d]^{\cong}&
H^*_{c,L}
(\bold{M}_{\Gamma\backslash C, v_1}
\times \bold{M}_{\Gamma\backslash C, v_2}\times \bold{M}_{C, v_1, v_2},\varphi_{\tr W_{v_1}\boxplus W_{v_2}})^\vee \ar[d]^{\cong}\\
H_{c,L}^*(\bold{J}_{\Gamma\backslash C, v_1}
\times \bold{J}_{\Gamma\backslash C, v_2}\times \bold{M}_{C, v_1}
\times \bold{M}_{C, v_2})^\vee\ar[r]^{\overline{p_1}^*}&
H_{c,L}^*(\bold{J}_{\Gamma\backslash C, v_1}
\times \bold{J}_{\Gamma\backslash C, v_2}\times \bold{M}_{C, v_1, v_2})^\vee
}\] Indeed, Lemma \ref{lem:pullback}, we take $X'=X=\bold{M}_{\Gamma\backslash C, v_1} \times \bold{M}_{\Gamma\backslash C, v_2}$, $g$ to be the identity map on $X$, and 
$h: \bold{M}_{C, v_1, v_2} \to \bold{M}_{C, v_1}\times \bold{M}_{C, v_2}$ 
the natural projection. Note that the map $p_1=g\times h$
is an affine bundle. Then Lemma~\ref{lem:pullback} implies that the vanishing cycle pullback $p_1^*$ is the same as $\overline{p_1}^*$. 

{\bf Step 2. }The map $p_3: \bold{M}_{\Gamma, v_1, v_2} \to 
 \bold{M}_{\Gamma\backslash C, v_1}\times  \bold{M}_{\Gamma\backslash C, v_2}
\times  \bold{M}_{C, v_1, v_2}$
is also an affine bundle, with $\tr W_{v_1, v_2}$ on the total space $
\bold{M}_{\Gamma, v_1, v_2}
$ obtained by pulling back of $\tr W_{v_1} \boxplus \tr W_{v_2}$ from 
$\bold{M}_{\Gamma, v_1}\times \bold{M}_{\Gamma, v_2}$. Hence, 
$j_{3*}: 
H^*_{c,L}
(\bold{M}_{\Gamma\backslash C, v_1}
\times \bold{M}_{\Gamma\backslash C, v_2}\times \bold{M}_{C, v_1, v_2}
,\varphi_{\tr W_{v_1}\boxplus W_{v_2}})^\vee
\to
H^*_{c,L}
(\bold{M}_{\Gamma\backslash C, v_1, v_2}\times \bold{M}_{C, v_1, v_2},\varphi_{\tr W_{v_1, v_2}})^\vee
$ is $p_3^*$ followed by the Euler class of $j_3$, that is, $ p_3^*=\frac{1}{e(j_3)} j_{3*}$. 
Plugging this into the equality $(i_3j_3)_*=((j_{2}\times\id_{\bold{M}_{C,v}})i_1)_*$, we get that $i_{3*} \circ p_3^*$ is the same as $\frac{1}{e(j_3)}\circ ((j_{2}\times\id_{\bold{M}_{C,v}})i_1)_*$.

{\bf Step 3. }The map $i_1$ is a section of an affine bundle. The function $\tr W$ on $\bold{M}_{\Gamma\backslash C, v_1}\times  \bold{M}_{\Gamma\backslash C, v_2} \times  \bold{M}_{C, v} $ is obtained by pulling back  $\tr W$ from $\bold{M}_{\Gamma\backslash C, v_1}\times  \bold{M}_{\Gamma\backslash C, v_2} \times  \bold{M}_{C, v_1, v_2}$.
Hence, we have the commutative diagram
\[
\xymatrix@R=1.5em @C=1em{
H^*_{c,L}
(\bold{M}_{\Gamma\backslash C, v_1}
\times \bold{M}_{\Gamma\backslash C, v_2}\times \bold{M}_{C, v_1, v_2}
,\varphi_{\tr W_{v_1}\boxplus W_{v_2}})^\vee
\ar[d]^{i_{1*}} \ar[r]^(.6){\cong}&
H^{\BM}_L( \bold{J}_{\Gamma\backslash C, v_1}
\times \bold{J}_{\Gamma\backslash C, v_2}\times  \bold{M}_{C, v_1, v_2})
\ar[d]^{\overline{i_1}_*}\\
H^*_{c,L}
(\bold{M}_{\Gamma\backslash C, v_1}
\times \bold{M}_{\Gamma\backslash C, v_2}\times \bold{M}_{C, v}
,\varphi_{\tr W_{v_1}\boxplus W_{v_2}})^\vee  \ar[r]^(.6){\cong} &H^{\BM}_{L}
(\bold{J}_{\Gamma\backslash C, v_1}
\times \bold{J}_{\Gamma\backslash C, v_2}\times \bold{M}_{C, v})
}
\]
since both $i_{1*}$ and  $\overline{i_1}_*$ are multiplication by the Euler class. 

Precomposing $i_{1*}=\overline{i_1}_*$ with equality from Lemma~\ref{lem:app_euler_phi_van}, we have $ e(\iota) ( j_{2}\times\id_{\bold{M}_{C,v}})_*i_{1*} =\omega^{\sharp}_{\overline{\omega}} e(j_{2}) \overline{i_1}_*$. 
This implies the equality 
$\frac{1}{e(\iota)}\omega_{\overline{\omega}}^{\sharp} \circ \overline{i_1}_*=\frac{1}{e(j_3)}\circ ((j_{2}\times\id_{\bold{M}_{C,v}})i_1)_*$, with the right hand side equal to  $i_{3*}p_3^*$ by Step 2. 

{\bf Step 4. }In Lemma \ref{lem:pushforward}, we take $X$ to be $\bold{M}_{\Gamma\backslash C, v_1, v_2}$, and $X'$ to be $\bold{M}_{\Gamma\backslash C, v_1+v_2}$, $h: \bold{M}_{C, v} \to \bold{M}_{C, v}$ the identity map. Then, 
$i_{2}=g\times h$. By Lemma \ref{lem:pushforward}, the Borel-Moore homology pushforward 
\[\overline{i_2}_*: H_{c,L}(\bold{J}_{\Gamma\backslash C, v_1, v_2}\times \bold{M}_{C, v}
)^\vee
\to H_{c,L}( \bold{J}_{\Gamma \backslash C, v_1+v_2}\times \bold{M}_{C, v})^\vee\] coincides with the vanishing cycle pushforward
\[
i_{2*}:H_{c, L}^*(\bold{M}_{\Gamma\backslash C, v_1, v_2}\times \bold{M}_{C, v}, 
\varphi_{\tr W_{v_1+ v_2} } )^{\vee}\to 
H_{c, L}^*(\bold{M}_{\Gamma, v_1+v_2}, 
\varphi_{\tr W_{v_1+v_2} } )^{\vee}.\]
\end{proof}

\subsection{Proof of Lemma~\ref{lem:app_euler_phi_van}}\label{subsec:proof of lemma}
Recall in \eqref{corresp with fiber}, we have the map 
$\omega:  \bold{M}_{\Gamma\backslash C, v_1, v_2}\to \bold{Y}$, given by $ x\mapsto ( ({\partial W}/{\partial a})_{a\in C} (x), \pr(x))$, with 
$\pr: \bold{M}_{\Gamma\backslash C, v_1, v_2} \to \bold{M}_{\Gamma\backslash C, v_1}\times \bold{M}_{\Gamma\backslash C, v_2}$ being the natural projection. We introduce a variety 
\[\bold{X}:=\{(l, x)\in \bold{M}_{C, v_1, v_2}\times 
\bold{M}_{\Gamma \backslash C,  v_1, v_2}
\mid
({\partial W}/{\partial a})_{a\in C} (\pr(x) )=\pr(l)\},\] where $\pr(l)$ is the image of $l$ under the projection $\pr: \bold{M}_{C, v_1, v_2} \to \bold{M}_{C, v_1}\times \bold{M}_{C, v_2}$.  
Now $\omega$ is equal to the composition of the following two maps
\[
a: \bold{M}_{\Gamma\backslash C, v_1, v_2} \to \bold{X}, x\mapsto (({\partial W}/{\partial a})_{a\in C} (x),  x), \text{and} \,\ 
 b: \bold{X} \to \bold{Y}, (l, x)\mapsto (l, \pr(x)),
\]
with $a$ being a closed embedding and $b$ being an affine bundle.
Using these notations,  we have the following diagram with $\omega$, $\overline{\omega}$, and $\iota$ the same as in diagram \eqref{corresp with fiber}.
\begin{equation}\label{eqn:iota}
\xymatrix@R=1em {
&\bold{X}   \ar[ld]^{b} &\\
\bold{Y}\ar@/^/@{.>}[ru]^{\tilde{j_2}} & & \bold{M}_{\Gamma\backslash C, v_1, v_2}\ar[lu]_{a} \ar[ll]_{\omega}\\
\bold{M}_{\Gamma\backslash C, v_1}\times \bold{M}_{\Gamma\backslash C, v_2}
 \ar@{^{(}->}[u]^{\iota}
\ar@{^{(}->}[urr]_(0.6){j_2}&  & \\
\bold{J}_{\Gamma\backslash C, v_1}\times \bold{J}_{\Gamma\backslash C, v_2}
\ar@{^{(}->}[u]
\ar@/^5.0pc/[uu]^{l}
&  & \bold{J}_{\Gamma\backslash C, v_1, v_2}\ar[ll]_(0.3){\overline{\omega}}\ar@{^{(}->}[uu]^{i}
}
\end{equation}
Here
$l: \bold{J}_{\Gamma\backslash C, v_1}\times \bold{J}_{\Gamma\backslash C, v_2}
\inj \bold{Y}$ is the embedding, and $\tilde{j_2}$ is the zero-section of the affine bundle $b$. 

\begin{lemma}\label{lem:app_euler}
Notations as above, under the isomorphism in Theorem~\ref{thm:DavA1}, the morphism
\[
(j_{2}\times\id_{\bold{M}_{C,v}})_*: H_{c, L}^*(\bold{M}_{\Gamma\backslash C, v_1}\times \bold{M}_{\Gamma\backslash C, v_2} \times \bold{M}_{C, v} , 
\varphi_{\tr W_{v_1+ v_2} } )^{\vee}\to 
H_{c, L}^*( \bold{M}_{\Gamma\backslash C, v_1, v_2}\times \bold{M}_{C, v} , 
\varphi_{\tr W_{v_1+v_2} } )^{\vee}\] is given by
$\mathbb{D} \circ (p_{\bold{J}_{v_1}\times \bold{J}_{ v_2}})_{!} l^* b_{!} a_! 
(\Q_{\bold{M}_{\Gamma\backslash C, v_1, v_2}}\to j_{2*}\Q_{\bold{M}_{\Gamma\backslash C, v_1}\times \bold{M}_{\Gamma\backslash C, v_2}})[-\dim \bold{M}_{C,v}]$. \footnote{Here we extend the constructions equivariently in the same way as in \cite{D}. That is, assume $G$ embeds into $\GL_{\mathbb{C}}(n)$, for some $n\in \mathbb{N}$. 
Let $EG_N$ be the space of $n$-tuples of linearly independent vectors in $\mathbb{C}^N$, which is denoted by $\hbox{fr}(n, N)$ in \cite{D}.
Thus, $EG_N/\GL_{\mathbb{C}}(n)$ is the Grassmannian of linear subspaces of dimension $n$ in $\mathbb{C}^N$.
Let $X_N:=EG_N\times_GX$ with natural maps $\{h_N :X_N\to X_{N+1}\}$. In what follows, $\bbQ_X$ should be understood as $\{\bbQ_{X_N}\}$ together with the natural maps $\{(h_{N})_{!}\bbQ_{X_N}\to \bbQ_{X_N+1}\}$. For an equivariant map $f:X\to Y$ between smooth varieties, we have the system of maps 
$\{(f_N)_!\bbQ_{X_N}\to \bbQ_{Y_N}[-\dim f]\}$ and $\{\bbQ_{Y_N}\to (f_N)_*\bbQ_{X_N} \}$,  together with commutative squares 
$$\xymatrix@C=1.5em @R=1.5em{(h_N)_!f_{N*}\bbQ_{X_N}\ar[r]&(f_{N+1})_* \bbQ_{X_{N+1}}\\
		(h_N)!\bbQ_{Y_N}\ar[r]\ar[u]&\bbQ_{Y_{N+1}}\ar[u]
	}, $$ and similarly for the $!$-version. We omit the system from the notations. For example, in the proof of Lemma \ref{lem:app_euler}, the morphism in each line should be understood as a system of morphisms. The equalities in the proof come from natural isomorphism of functors, hence give rise to equalities of commutative squares.}
\end{lemma}
\begin{proof}
Let $\pi: \bold{M}_{\Gamma\backslash C, v_1, v_2}\times  \bold{M}_{C, v} \to \bold{M}_{\Gamma\backslash C, v_1, v_2}$ be the projection. 
By Theorem~\ref{thm:DavA1},  $(j_{2}\times\id_{\bold{M}_{C,v}})_*$ is
\begin{align*}
&\mathbb{D} \circ (p_{\bold{M}_{\Gamma\backslash C, v_1, v_2}})_{!} 
\pi_{!} \pi^* i_{!} i^*(\Q_{\bold{M}_{\Gamma\backslash C, v_1, v_2}}\to j_{2*}\Q_{\bold{M}_{\Gamma\backslash C, v_1}\times \bold{M}_{\Gamma\backslash C,  v_2}})\\
=&\mathbb{D} \circ (p_{\bold{M}_{\Gamma\backslash C, v_1, v_2}})_{!}  i_{!} i^*(\Q_{\bold{M}_{\Gamma\backslash C, v_1, v_2}}\to j_{2*}\Q_{\bold{M}_{\Gamma\backslash C, v_1}\times \bold{M}_{\Gamma\backslash C,  v_2}})[-\dim \bold{M}_{C,v}]\\
=&\mathbb{D} \circ (p_{\bold{J}_{v_1}\times \bold{J}_{ v_2}})_{!}  \overline{\omega}_{!} i^*(\Q_{\bold{M}_{\Gamma\backslash C, v_1, v_2}}\to j_{2*}\Q_{\bold{M}_{\Gamma\backslash C, v_1}\times \bold{M}_{\Gamma\backslash C,  v_2}})[-\dim \bold{M}_{C,v}]\\
=&\mathbb{D} \circ (p_{\bold{J}_{v_1}\times \bold{J}_{ v_2}})_{!}  l^* \omega_{!} (\Q_{\bold{M}_{\Gamma\backslash C, v_1, v_2}}\to j_{2*}\Q_{\bold{M}_{\Gamma\backslash C, v_1}\times \bold{M}_{\Gamma\backslash C,  v_2}})[-\dim \bold{M}_{C,v}]\\
=&\mathbb{D} \circ (p_{\bold{J}_{v_1}\times \bold{J}_{ v_2}})_{!}  l^* b_{!}a_{!} (\Q_{\bold{M}_{\Gamma\backslash C, v_1, v_2}}\to j_{2*}\Q_{\bold{M}_{\Gamma\backslash C, v_1}\times \bold{M}_{\Gamma\backslash C,  v_2}})[-\dim \bold{M}_{C,v}].
\end{align*}
Here the first isomorphism follows from projection formula (see for example \cite[Theorem 2.3.29]{Di}) and the fact that $\pi_!\bbQ=\bbQ[-\dim \bold{M}_{C,v}]$. 
\end{proof}

\begin{lemma}\label{lem:Gysin pullback of phi}
The Gysin pullback $\omega^{\sharp}_{\overline{\omega}}: 
H^{\BM}_{L}( \bold{J}_{\Gamma\backslash C, v_1}\times \bold{J}_{\Gamma\backslash C, v_2}  )
\to H^{\BM}_{L}(\bold{J}_{\Gamma\backslash C, v_1, v_2})$ is given by 
$\bbD\circ(p_{\bold{J}_{v_1}\times \bold{J}_{v_2}})_! \circ l^*$ applied to the following composition of morphisms
\[
b_!a_! \Q_{\bold{M}_{\Gamma\backslash C, v_1, v_2}} \to b_! \Q_{\bold{X}}[-\dim a] \cong \Q_{\bold{Y}}[-\dim a-\dim b]. 
\]
\end{lemma}
\begin{proof}
Let $\bold{X}' \subset \bold{X}$ be the inverse image of $\bold{J}_{\Gamma\backslash C, v_1}\times \bold{J}_{\Gamma\backslash C, v_2}$ under the map $b$. Then, we have the following diagram with $l'$ induced by $l$.
\begin{equation}\label{eqn:diag_phi}
\xymatrix@R=1em{
\bold{Y} & \bold{X}  \ar[l]_{b}& \bold{M}_{\Gamma\backslash C, v_1, v_2}
 \ar[l]_(0.6){a} \\
\bold{J}_{\Gamma\backslash C, v_1}\times \bold{J}_{\Gamma\backslash C, v_2}  \ar@{^{(}->}[u]^{l}& \bold{X}' \ar[l]_(0.3){b'}\ar@{^{(}->}[u]^{l'}
& \bold{J}_{\Gamma\backslash C, v_1, v_2}\ar[l] \ar@{^{(}->}[u]
\ar@/^1.0pc/[ll]|-{\overline{\omega}}
}
\end{equation}

It is well-known (see, e.g., \cite[Lemma 2.1.2]{KSa}) that  
 the map $a^{\sharp}$ is given by $(p_{\bold{X}'})_{*}l'^{!}(\Q_{\bold{X}}\to a_*\Q_{\bold{M}_{\Gamma\backslash C, v_1, v_2}} )$. 
We have the following equalities
\begin{align*}
(p_{\bold{X}'})_{*}l'^{!}(\Q_{\bold{X}}\to a_*\Q_{\bold{M}_{\Gamma\backslash C, v_1, v_2}} )
=&(p_{\bold{J}_{v_1}\times \bold{J}_{v_2}})_{*} b'_{*} l'^{!}(\Q_{\bold{X}}\to a_*\Q_{\bold{M}_{\Gamma\backslash C, v_1, v_2}} )\\
=& (p_{\bold{J}_{v_1}\times \bold{J}_{v_2}})_{*}  l^{!}b_{*}(\Q_{\bold{X}}\to a_*\Q_{\bold{M}_{\Gamma\backslash C, v_1, v_2}} )\\
=& \mathbb{D}\circ(p_{\bold{J}_{v_1}\times \bold{J}_{v_2}})_{!}  l^{*}b_{!}(a_!\Q_{\bold{M}_{\Gamma\backslash C, v_1, v_2}}\to \Q_{\bold{X}}[-\dim a] )
\end{align*}

The map $b$ is an affine bundle. Hence, $b^\sharp$ is induced by $b_!\bbQ_{\bold{X}}\cong\bbQ_{\bold{Y}}[-\dim b]$. The assertion now follows from $\omega^{\sharp}_{\overline{\omega}} =a^\sharp\circ b^\sharp$.
\end{proof}

Notations as in \eqref{diag:j23} and \eqref{eqn:iota}, $e(a)$ and $e(j_{2})$ are the Euler classes of the two embeddings $a$ and $j_2$. Note that $e(j_2)=e(j_3)$ and $e(a)=e(\iota)$. Combining Lemmas~\ref{lem:app_euler} and \ref{lem:Gysin pullback of phi}, we get Lemma~\ref{lem:app_euler_phi_van}.

\begin{proof}[Proof of Lemma~\ref{lem:app_euler_phi_van}]
By Lemma \ref{lem:app_euler}, $e(a) (j_{2}\times\id_{\bold{M}_{C,v}})_*$ is the same as applying $\mathbb{D} \circ (p_{\bold{J}_{v_1}\times \bold{J}_{ v_2}})_{!}  l^*$ to the following composition 
\begin{align*}
 b_{!}
 \Big( &a_{!} \Q_{\bold{M}_{\Gamma\backslash C, v_1, v_2}} \to \Q_{\bold{X}}[-\dim a] \to a_{*} \Q_{\bold{M}_{\Gamma\backslash C, v_1, v_2}}[-\dim a]
=a_! \Q_{\bold{M}_{\Gamma\backslash C, v_1, v_2}}[-\dim a]\\& \to a_{!}j_{2*} \Q_{ \bold{M}_{\Gamma\backslash C, v_1}\times \bold{M}_{\Gamma\backslash C, v_2}}[-\dim a] \Big).
\end{align*}
By Lemma \ref{lem:Gysin pullback of phi}, $\omega^{\sharp}_{\overline{\omega}} e(j_{2})$ is $\mathbb{D} \circ (p_{\bold{J}_{v_1}\times \bold{J}_{ v_2}})_{!}  l^*$ applied to the composition 
\[
b_! a_! \Q_{\bold{M}_{\Gamma\backslash C, v_1, v_2}} \to \Q_{\bold{Y}}[-\dim a-\dim b]=b_! \tilde{j}_{2!} \Q_{\bold{Y}}[-\dim a-\dim b] \to 
b_! \Q_{\bold{X}}[-\dim a] \to b_! \tilde{j}_{2*}\Q_{\bold{Y}} [-\dim a].
\]
We have the commutative diagram
\[\xymatrix@R=1em @C=1em{
a_! \Q_{\bold{M}_{\Gamma\backslash C, v_1, v_2}}[\dim a]\ar[r]&\Q_{\bold{X}}\ar[r]& a_{*}\Q_{\bold{M}_{\Gamma\backslash C, v_1, v_2}}\ar[r]& a_! j_{2!}\Q_{\bold{M}_{\Gamma\backslash C, v_1}\times \bold{M}_{\Gamma\backslash C, v_2}}\\
a_! \Q_{\bold{M}_{\Gamma\backslash C, v_1, v_2}}[\dim a]\ar[r]\ar@{=}[u]&\tilde{j}_{2!} \Q_{\bold{Y}}[-\dim b]\ar[u]\ar[r] &\Q_{\bold{X}}\ar[r]\ar[u] & \tilde{j}_{2!}\Q_{\bold{Y}}\ar[u].
}\]
Applying the functor $(p_{\bold{J}_{v_1}\times \bold{J}_{ v_2}})_{!}  l^*b_{!}[-\dim a]$ to the diagram above  gives the desired equality. 
\end{proof}

\section{The preprojective cohomological Hall algebra}
\label{preproj CoHA}
This section is a brief review of the main results about the preprojective CoHA from \cite{YZ15}. All the details can be found in {\it loc. cit.}.
\subsection{The preprojective CoHA}\label{subsec:HallMulti}
Let $Q=(I, H)$ be a quiver, where $I$ is the set of vertices, and $H$ the set of arrows. 
Let $\overline{Q}$ be the double quiver, and $\Pi_Q$ be the preprojective algebra of $Q$. 
By definition, the vertices of $\overline{Q}$ are the same as the vertices of $Q$, and the arrows of $\overline{Q}$ are the arrows of $Q$ together with arrows $h^*: j\to i$ for each arrow $h: i\to j$ of $Q$. The preprojective algebra $\Pi_Q$ is defined as the quotient of the path algebra of $\overline{Q}$ by the relation $\sum_{h\in H}[h, h^*]$. 

We have the following moment map on $T^*\bold{M}_{Q,v}=\bold{M}_{\overline{Q},v}$. 
\[\mu_{v}: T^*\bold{M}_{Q,v}\to \fg^*_{v}, \,\ (x, x^{*})\mapsto [x, x^{*}]. \] 

Denote by $\bold{\Lambda}_{v}$ the preimage of $0$ under the map $\mu_{v}$. 
The subvariety $\bold{\Lambda}_{v} \subset \bold{M}_{\overline{Q}, v}$ parametrizes the representations of $\Pi_Q$ of dimension vector $v$.  

We fix a weight function $m:H\coprod H^*\to\bbZ$, and define the
 torus $T=\Gm^2$ action on $T^*\bold{M}_{Q, v}$ 
 as follows (see also  \cite[(2.7.1) and (2.7.2)]{Nak99}).
For $h\in H$ from $i$ to $j$, the element  $(t_1,t_2)\in T$ acts on the factor  $\Hom(V^i, V^j)$ by $t_1^{m_h}$, and on the factor  $\Hom(V^j, V^i)$ by $t_2^{m_{h^*}}$.
We let $T$ act on the Lie algebra $\fg_v^*$ of $G_v$ by weight $t_1t_2$.
It induces a $T$-action on $\bold{\Lambda}_{v}$ if the moment map is $T$-equivariant, or equivalently, if we have the following (\cite[Assumption~3.1]{YZ15}).
\begin{assumption} \label{Assu:WeghtsGeneral}
We assume the weight functions $m$ and a subtorus $D\subseteq T$  are compatible, in the sense that $t_1^{m_h}t_2^{m_{h^*}}=t_1t_2$ on $D$, for all $h$. 
\end{assumption}
One example of  $T$ action satisfying Assumption~\ref{Assu:WeghtsGeneral} is the following. 
\begin{example}\label{ex:two t and hbar}
Let $D=\Gm$ with coordinate $\hbar$, sitting in $T$ via $t_1=t_2=\hbar/2$. For any pair of vertices $i$ and $j$, label the arrows from $i$ to $j$ by $h_1, \dots, h_a$.  The pairs of integers are $m_{h_p}=a+2-2p$ and $m_{h_p^*}=-a+2p$. 
\end{example}

Let $A$ be an oriented cohomology theory with $A(\pt)=R$. We consider the $\bbN^I$-graded $R[\![t_1, t_2]\!]$-module 
 \[
\calP(A, Q):=\bigoplus_{v}\calP_v(A, Q)=\bigoplus_{v} A_{G_v\times T}(\bold{\Lambda}_{v}).
\]
For each pair $v_1,v_2\in\bbN^I$, the multiplication map $m_{v_1, v_2}^{\prepr}:\calP_{v_1}\otimes_{\mathbb{Q}[\![t_1t_2]\!]}\calP_{v_2}\to \calP_{v_1+v_2}$ is as follows. 

Recall that $\bold{M}_{Q, v_1, v_2}=\{x\in \bold{M}_{Q, v_1+v_2}\mid x(V_1)\subset V_1\}$. We have the following correspondence of $G\times T$-varieties
\begin{equation}\label{eq:corr local}
G\times_P(\bold{M}_{Q, v_1}\times \bold{M}_{Q, v_2}) \leftarrow 
G\times_P \bold{M}_{Q, v_1, v_2} \rightarrow 
\bold{M}_{Q, v}. 
\end{equation}
Let $\bold{\Lambda}_{v_1, v_2}:=\bold{\Lambda}_{v}\cap \bold{M}_{\overline{Q}, v_1, v_2}
= \{(x, x^*)\in \bold{M}_{\overline{Q}, v_1, v_2} \mid [x, x^*]=0\}$, and 
$X=G\times_P(\bold{M}_{{Q}, v_1}\times \bold{M}_{{Q}, v_2})$. By \cite[Lemma 5.1]{YZ15}, we have
\[
T^*X=G\times_P \{ (c, x, x^*) \in \fp_{v_1, v_2}\times (\bold{M}_{{Q}, v_1}\times \bold{M}_{{Q}, v_2})\times 
(\bold{M}_{{Q^{op}}, v_1}\times \bold{M}_{Q^{op}, v_2})
\mid 
[x, x^*]=\pr(c)
  \}, 
\]
where $\pr: \fp_{v_1, v_2}\to \fg_{v_1}\times \fg_{v_2}$ is the natural projection.

By the standard formalism of Lagrangian  correspondences (see, e.g.,  \cite[\S~7]{SV2} and \cite[\S~1.4]{YZ15}), the correspondence \eqref{eq:corr local} induces the 
following commutative diagram of $G\times T$-varieties.
\begin{equation}\label{diag:Lag}
\xymatrix@R=1.5em{
 G\times_P(\bold{\Lambda}_{v_1}\times \bold{\Lambda}_{v_2})
\ar@{^{(}->}[d]
&&
G\times_P \bold{\Lambda}_{v_1, v_2}
\ar[ll]_{\overline\phi}\ar[r]^(0.6){\overline\psi}\ar@{^{(}->}[d]&
\bold{\Lambda}_{v}\ar@{^{(}->}[d]\\
G\times_P(\bold{M}_{\overline{Q}, v_1}\times \bold{M}_{\overline{Q}, v_2})
\ar@{^{(}->}[r]&T^*X&
G\times_P\bold{M}_{\overline{Q}, v_1, v_2}
\ar[l]_(0.6){\phi}\ar[r]^(0.6){\psi}&\bold{M}_{\overline{Q}, v}
}
\end{equation}
Note that the left square of \eqref{diag:Lag}  is a Cartesian square.

The  map $m_{v_1, v_2}^{\prepr}$ is defined to be the composition of the following morphisms. 
\begin{enumerate}
\item The K\"unneth morphism $
 \calP_{v_1}\otimes_{R[\![t_1,t_2]\!]}\calP_{v_2}\to 
A_{L\times T}(\bold{\Lambda}_{v_1}\times \bold{\Lambda}_{v_2}). 
$
\item The natural  isomorphism
$
A_{L \times T}(\bold{\Lambda}_{v_1}\times \bold{\Lambda}_{v_2})\cong A_{G \times T}\big(G\times_P(\bold{\Lambda}_{v_1}\times \bold{\Lambda}_{v_2})\big).$
\item The following composition from diagram \eqref{diag:Lag}, with $\phi^\sharp_{\overline{\phi}}$ being the Gysin pullback of $\phi$
\[
\xymatrix{
A_{G \times T}\big(G\times_P(\bold{\Lambda}_{v_1}\times \bold{\Lambda}_{v_2})\big) 
\ar[r]^(0.6){\phi^{\sharp}_{\overline{\phi}}} &A_{G\times T}(G\times_P \bold{\Lambda}_{v_1, v_2}) 
\ar[r]^(0.55){\overline\psi_*} &A_{G\times T}(\bold{\Lambda}_{v}) \cong \calP_{v}.
} \]
\end{enumerate}

\subsection{Representations from quiver varieties}
\label{subsec:PreproRepn}
Let $Q^{\heartsuit}$ be the framed quiver. Recall that the set of vertices of $Q^{\heartsuit}$ is $I \sqcup I'$, with a bijection $I\to I'$ sending $i$ to $i'$. The set of edges of $Q^\heartsuit$ is, by definition, a disjoint union of $H$ and a set of additional edges $j_i : i \to i'$, one for each vertex $i\in I$, $i'\in I'$. Let $\overline{Q^{\heartsuit}}$ be the double of $Q^{\heartsuit}$.

Let $w\in \N^I$ be the framing of $Q^{\heartsuit}$. Let $V, W$ be two $I$-tuples of vector spaces with dimension vectors $v, w\in \bbN^I$. Let $\bold{M}_{\overline{Q^\heartsuit},(v, w)}$ be the representation space of the double framed quiver $\overline{Q^\heartsuit}$. 
We have the isomorphism
\[
\bold{M}_{\overline{Q^{\heartsuit}}, v, w} 
\cong \bold{M}_{\overline{Q}, v} \oplus \Hom_{Q}(V, W) \oplus \Hom_{Q}(W, V).
\]
For a fixed $I$-tuple of subvector spaces $V_1\subset V$, such that $\dim(V_1)=v_1$, recall that we have
$\bold{M}_{\overline{Q}, (v_1, v_2)}=\{b\in \bold{M}_{\overline{Q}, v_1+v_2}\mid b(V_1)\subset V_1\}$. 
Define
\[
\bold{M}_{\overline{Q^{\heartsuit}}, (v_1, v_2, w)}:= \{(b, i, j) \mid b \in  
\bold{M}_{\overline{Q}, (v_1, v_2)}, i\in \Hom(W, V), 
j\in \Hom(V, W),   \Image(i)\subset V_1\}.
\]
It gives the following correspondence of $G_v\times G_w\times T$-varieties
\begin{equation}\label{equ:corresp 2}
G_v\times_P( \bold{M}_{\overline{Q^\heartsuit},(v_1, w)}\times \bold{M}_{\overline{Q},v_2} 
)\leftarrow 
 G_v\times_P\bold{M}_{\overline{Q^\heartsuit}, (v_1, v_2, w)}\rightarrow 
\bold{M}_{\overline{Q}^\heartsuit,(v, w)}.
\end{equation}
Let $\mu_{v, w}: \bold{M}_{\overline{Q^\heartsuit},(v, w)}\to \fg_{v}^*$ be the moment map for the $G_v$-action. 
Denote by $\bold{\Lambda}_{v, w}$ the preimage of $0$ under $\mu_{v, w}$, and let $\bold{\Lambda}_{v_1, v_2, w}:=\bold{\Lambda}_{v_1+v_2, w}\cap \bold{M}_{\overline{Q^{\heartsuit}}, (v_1, v_2, w)}$ be the intersection. 

Consider the stability condition $\theta^+=(1, 1, \cdots, 1)$ induced by the character 
$G_v\to \Gm: (g_i)_{i\in I} \mapsto \prod_{i\in I} \det(g_i)^{-1}$. 
Let $\bold{M}_{\overline{Q^\heartsuit},(v, w)}^{ss}$ be the set of semistable points of $\bold{M}_{\overline{Q^\heartsuit},(v, w)}$ under the $G_v$-action. The point $(x, x^*, i, j)\in \bold{\Lambda}_{v, w}$ is $\theta^+$--semistable (\cite[Corollary 5.1.9]{G}), if and only if, the following holds:
For any collection of vector subspaces $S = (S_i)_{i\in k}\subset
V = (V_i)_{i\in k}$,  which is stable under the maps
$x$ and $x^*$, if
$S_k\subset \ker(j_k)$ for any $k\in I$, then $S=0$. Note that this choice of stability condition is essential in constructing a right action of $\calP$ (see \cite[Remark~5.5]{YZ15} for details).

Consider the variety $X=G\times_P(\bold{M}_{Q^{\heartsuit},(v_1, w)}
\times \bold{M}_{Q, v_2})$. 
There is a bundle projection $T^*X\to G\times_P(\bold{M}_{\overline{Q^{\heartsuit}},(v_1, w)}
\times \bold{M}_{\overline{Q}, v_2})$. We define $T^*X^s$ to be the preimage of $G\times_P(\bold{M}_{\overline{Q^{\heartsuit}},(v_1, w)}^{ss}\times \bold{M}_{\overline{Q}, v_2})$ under this bundle projection. By \cite[Lemma 5.1]{YZ15}, we have
\[
T^*X^s=G\times_P\{(c, (x, x^*, i, j), (y, y^*))
\in \fp_{v_1, v_2}\times \bold{M}_{\overline{Q^{\heartsuit}}, v_1, w}^{ss}\times \bold{M}_{\overline Q, v_2}
\mid 
 [x, x^*]+i\circ j=\pr_1(c), [y, y^*]=\pr_2(c)\}, 
\]
where $\pr_i:  \fp_{v_1, v_2}\to \fg_{v_i}$ is the natural projection, $i=1, 2$.
The correspondence \eqref{equ:corresp 2} induces the 
following commutative diagram of $G_v\times T\times G_w$-varieties.
\begin{equation} \label{corr for action}
\xymatrix@R=1.5em{
G\times_{P}(\bold{\Lambda}_{v_1, w}^{ss}\times \bold{\Lambda}_{v_2})
\ar@{^{(}->}[d]
&&G\times_P\bold{\Lambda}_{v_1, v_2, w}^{ss} \ar@{^{(}->}[d]\ar[ll]_{\overline\phi}\ar[r]^(0.4){\overline\psi}&
\bold{\Lambda}_{v, w}^{ss}
\ar@{^{(}->}[d]\\
G\times_P(\bold{M}_{\overline{Q^\heartsuit},(v_1, w)}^{ss}\times \bold{M}_{\overline{Q}, v_2})\ar@{^{(}->}[r]^(0.8){\iota} & T^*X^s &
{G\times_P \bold{M}_{\overline{Q^{\heartsuit}}, (v_1, v_2, w)}^{ss}} \ar[l]_(.7)\phi \ar[r]^(.6)\psi & \bold{M}_{\overline{Q^\heartsuit},(v, w)}^{ss}.
}\end{equation}
Here the left square of diagram \eqref{corr for action} is a pullback diagram.

For any dimension vectors $v$, $w\in \bbN^I$, we have the isomorphism 
\[\calM(v,w):=A_{T\times G_w}(\fM(v, w))\cong A_{G_v\times T\times G_w}(\bold{\Lambda}_{v, w}^{ss}).\]
The action of $\calP(A, Q)$ on $\calM(w):=\bigoplus_{v}\calM(v,w)$, as a map
\[
a_{v_1,v_2}:  \calM(v_1,w) \otimes \mathcal{P}_{v_2} \to
\calM(v_1+v_2,w)\,\  \text{for each pair $v_1,v_2\in\bbN^I$ },
\]
is defined to be the composition of the K\"unneth morphism
\begin{align}\label{equ: action iso}
A_{G_{v_1}\times G_w \times T}(\bold{\Lambda}_{v_1, w}^{ss}) 
\otimes A_{G_{v_2}\times T}(\bold{\Lambda}_{v_2}) 
\to &A_{L\times T\times G_w}
(\bold{\Lambda}_{v_1, w}^{ss}\times \bold{\Lambda}_{v_2}),
\end{align}
 with the morphism
\[
\overline\psi_*\circ\phi^{\sharp}_{\overline{\phi}}:A_{G_{v}\times T\times G_w}\left(
G\times_P(\bold{\Lambda}_{v_1, w}^{ss}\times \bold{\Lambda}_{v_2})
\right)\to A_{G_{v}\times T\times G_w}(
\bold{\Lambda}_{v, w}^{ss})=\calM(v, w).
\]
Here  $\phi^{\sharp}_{\overline{\phi}}$ is the refined Gysin pullback of $\phi$ in diagram \eqref{corr for action}.

\section{Preprojective CoHA vs critical CoHA}
\label{sec:aux CoHA}
We compare the preprojective CoHA of quiver $Q$ with the auxiliary CoHA introduced in \S \ref{sec:aux CoHA in sec1} applied to a special quiver $\widehat{Q}$ with potential $W$ as below.
\subsection{A special case of potential and cut}
\begin{example}\label{exam:GuinAlg}
Let $Q=(I, H)$ be any quiver. 
Let  $\Gamma$ be the extended quiver $\widehat{Q}$
introduced by Ginzburg in \cite{G}.
More precisely,  $\widehat{Q}$ has the same set of vertices as $Q=(I, H)$, and the following set of arrows:
\begin{enumerate}
\item an arrow $a: i\to j$ for any arrow $a: i\to j$ in $Q$,
\item an arrow $a^*: j\to i$ for any arrow $a: i\to j$ in $Q$,
\item a loop $l_i:i\to i$ for any vertex $i$ in $Q$.
\end{enumerate}
Define a potential $W$ on $\widehat{Q}$ by the formula
\[ 
W=\sum_{(a: i\to j)\in H} (l_j a a^*-l_i a^* a)=
\sum_{i\in I} l_i \cdot \sum_{a\in H}[a, a^*].
\]
Let $C=\{l_i \mid i\in I\}$ be the cut of the pair $(\widehat{Q}, W)$. In this case, the space $\bold{J}_{\Gamma\backslash C, v}$ is the representation space of the preprojective algebra $\Pi_{Q}:=\C\overline Q/(\sum_{a\in H} [a, a^*])$. And for any $v\in\bbN^I$ we have the equality
\[
\bold{J}_{\Gamma\backslash C, v} \times \bold{M}_{C, v}=
\bold{\Lambda}_{v}\times  \fg_{ v}.
\]
\end{example}
\begin{example}\label{ex:Q heart}
Another example of the quiver with potential is $(\widehat{Q}^{\spadesuit}, W^{\spadesuit})$. 
Define a new quiver $\widehat{Q}^{\spadesuit}$ to have the same set of vertices as $Q^\heartsuit$ and the following arrows:
\begin{enumerate}
\item an arrow $a: i\to j$ for any arrow $a: i\to j$ in $Q^\heartsuit$,
\item an arrow $a^*: j\to i$ for any arrow $a: i\to j$ in $Q^\heartsuit$,
\item a loop $l_i:i\to i$ for any vertex $i$ in $Q$.
\end{enumerate}
Let $C=\{l_i\mid i\in I\}$ be the cut and $W^{\spadesuit}
=\sum_{k\in I} l_k \cdot \big(\sum_{a\in H}[a, a^*] +i_k  j_k\big)$ be the potential on $\widehat{Q}^{\spadesuit}$. 
In this case, the space $\bold{J}_{\Gamma \backslash C, v}$ is the representation space of $\C \overline{Q^{\heartsuit}}/(\sum_{a\in H} [a, a^*]+i j)$. Hence, we have the equality
\[
\bold{J}_{\Gamma \backslash C, v} \times \bold{M}_{C, v}
= \bold{\Lambda}_{v, w}\times \fg_{ v}.
\] The natural projection $\pi: \bold{\Lambda}_{v, w}\times \bold{M}_{C, v} \to  \bold{\Lambda}_{v, w}$ is a $G_{v}\times G_{w}$-equivariant vector bundle. We consider $\pi^{-1}(\bold{\Lambda}_{v, w}^{ss})= \bold{\Lambda}_{v, w}^{ss}\times \bold{M}_{C, v}$. Then $(\bold{\Lambda}_{v, w}^{ss}\times \bold{M}_{C, v})/G_{v}$ is an equivariant vector bundle on the quiver variety $\mathfrak{M}(v, w)=\bold{\Lambda}_{v, w}^{ss}/G_{v}$.
\end{example}
\subsection{The auxiliary CoHA in the case when $\Gamma=\widehat{Q}$}
\label{ex:ExtendedFramed}

Recall in \S \ref{subsec:AuxCoHA}, the {\it auxiliary cohomological Hall algebra}  associated to the data 
$(\widehat{Q}, W, C)$ and arbitrary cohomology theory $A$ is
 \[
\calH^{\aux}(\widehat{Q}, W, C)=
\bigoplus_{v\in \N^{\Gamma_0} } A_{G_v\times T} 
(\bold{\Lambda}_{v}\times \fg_v, \Q).
 \]
 Similar to \eqref{corresp with fiber} and \eqref{diag:Lag}, we have the following correspondence.
\begin{equation}
\label{equ:corres of crit coha}
\xymatrix@R=1em @C=0.5em{
G\times_{P}(\bold{M}_{\widehat{Q}, v_1}
\times \bold{M}_{\widehat{Q}, v_2})
\ar@{^{(}->}[r]
&G\times_{P} (\bold{Y}\times  \fg_{v_1}\times \fg_{v_2}) & 
G\times_P\bold{M}_{\widehat{Q}, v_1, v_2}
\ar[l]_(0.4){q} \ar[r]^(0.55){\hat\eta} &
\bold{M}_{\widehat{Q}, v_1+v_2}\\
G\times_{P}(
\bold{\Lambda}_{v_1}\times \fg_{v_1}\times
\bold{\Lambda}_{v_2}\times \fg_{v_2})
\ar@{^{(}->}[u]^{i_1}
& &
G\times_P(\bold{\Lambda}_{v_1, v_2}\times \fp_{v_1, v_2})
\ar[ll]_{\overline{q}}\ar[r]^(0.7){\overline {\hat\eta}} \ar@{^{(}->}[u]^{i_2} &
\bold{\Lambda}_{v}\times \fg_{v}
\ar@{^{(}->}[u]^{i_3}.
}\end{equation} 
Here the map $\hat{\eta}$ is defined by $(g, x)\mapsto g x g^{-1}$. The map $q$ is given by the product of 
$\omega: \bold{M}_{\overline{Q}, v_1, v_2} \to \bold{Y}$ (in diagram \eqref{corresp with fiber})
and the natural projection $\fp_{v_1, v_2}\to \fg_{v_1}\times  \fg_{v_2}$. Recall that $T$ acts on the Lie algebra $\fg_v^*$ by weight $t_1t_2$. 
Consequently,  the action of $T$ on $\fg_v$ is such that both $\Gm$-factors of $T$ have weight -1. We assume the extra $T$-action on $\bold{\Lambda}_{v}\times \fg_{v}$ restricted to $\bold{\Lambda}_{v}$ satisfies Assumption~\ref{Assu:WeghtsGeneral}, so that the correspondence is equivariant. 

The Hall multiplication $m^{\aux}_{v_1,v_2}$ of $\calH^{\aux}(\widehat{Q}, W, C)$ defined in \S \ref{subsec:AuxCoHA} is the same as the composition of the following morphisms.
\begin{enumerate}
\item The K\"unneth morphism
\begin{align*}
&A_{G_{v_1}\times T}(\bold{\Lambda}_{v_1}\times \fg_{v_1}) \otimes 
A_{G_{v_2}\times T}(\bold{\Lambda}_{v_2}\times \fg_{v_2})
\to
A_{L\times T}(\bold{\Lambda}_{v_1}\times \fg_{v_1}\times\bold{\Lambda}_{v_2}\times \fg_{v_2}).
\end{align*}
\item The isomorphisms:
\[
A_{L\times T}(\bold{\Lambda}_{v_1}\times \fg_{v_1}\times\bold{\Lambda}_{v_2}\times \fg_{v_2})
\cong 
A_{G\times T}(G\times_{P}(\bold{\Lambda}_{v_1}\times \fg_{v_1}\times\bold{\Lambda}_{v_2}\times \fg_{v_2}).
\]
\item 
The refined Gysin pullback along $q$ in \eqref{equ:corres of crit coha}:
\begin{align*}
q^\sharp_{\overline{q}}: 
A_{G\times T}(G\times_{P}(\bold{\Lambda}_{v_1}\times \fg_{v_1}\times\bold{\Lambda}_{v_2}\times \fg_{v_2})) \to 
A_{G\times T}\big(G\times_P (\bold{\Lambda}_{v_1, v_2}\times \fp_{v_1, v_2}) \big).
\end{align*}
\item
The pushforward $\overline {\hat\eta}_*$ in \eqref{equ:corres of crit coha}:
\begin{align*}
\overline {\hat\eta}_*:  A_{G\times T}\big(G\times_P (\bold{\Lambda}_{v_1, v_2}\times \fp_{v_1, v_2})\big)
\to 
A_{G\times T}(\bold{\Lambda}_{v}\times \fg_{v}).
\end{align*}
\end{enumerate}

\subsection{Action on the cohomology of quiver varieties}
Similar to \S\ref{subsec:PreproRepn}, we show $\calH^{\aux}(\widehat{Q}, W, C)$ acts on the equivariant $A$-homology of the Nakajima quiver varieties. As in \S\ref{ex:ExtendedFramed}, we will also take the $T$-action into consideration. 

Notations as in \S\ref{subsec:PreproRepn}. For any $v\in\bbN^I$, we have the isomorphism 
\[\calM(v, w):=A_{G_{w}\times T}(\mathfrak{M}(v, w))\cong A_{G_v\times 
G_{w}\times T}(\bold{\Lambda}_{v, w}^{ss}\times \fg_v).\]
For any $v_1,v_2, w \in \bbN^I, v=v_1+v_2$, we define a map
\[
a_{v_1, v_2}^{\aux}: \calM(v_1, w) \otimes \calH^{\aux}_{v_2}(\widehat{Q}, W, C) \to \calM(v, w).
\]
Let $V, W$ be two $I$-tuple of vector spaces with dimension vectors $v, w\in \bbN^I$. Recall we have the isomorphism
\[
\bold{M}_{\widehat{Q}^{\spadesuit}, v, w} 
\cong 
\bold{M}_{\overline{Q^{\heartsuit}}, (v, w)} \times \fg_v
\cong \bold{M}_{\overline{Q}, v} \times \Hom_{Q}(V, W) \times \Hom_{Q}(W, V) \times \fg_v.
\]
We start with the correspondence
\[
\xymatrix
{\bold{M}_{\widehat{Q}^{\spadesuit}, (v_1, w)}   \times 
\bold{M}_{\widehat{Q}, v_2}
& \bold{M}_{\widehat{Q}^{\spadesuit}, (v_1, v_2, w)}
\ar[l]_(0.4){p}\ar[r]
&
\bold{M}_{\widehat{Q}^{\spadesuit}, (v, w) }
},\]
where 
$
\bold{M}_{\widehat{Q}^{\spadesuit}, (v_1, v_2, w)}:= \{(b, i, j) \mid b \in  
\bold{M}_{\widehat{Q}, (v_1, v_2)}, i\in \Hom(W, V), 
j\in \Hom(V, W),   \Image(i)\subset V_1\}.
$
For $(b, i, j) \in \bold{M}_{\widehat{Q}^{\spadesuit}, (v_1, v_2, w)}$, 
denote by $(\pr_1(b), \pr_2(b))$ the projection of $b$ to $\bold{M}_{\widehat{Q}, v_1}\times \bold{M}_{\widehat{Q}, v_2}$.  Let $i_{V_1}: W\to V_1$ be the co-restriction of $i$ on $V_1$, and $j_{V_1}$ the restriction of $j: V\to W$ on $V_1$. The map $p$ is defined to be $p: (b, i, j) \mapsto (\pr_1(b), \pr_2(b), j_{V_1}, i_{V_1})$.

Let $\pi_i: \fp_{v_1, v_2}\to \fg_{v_i}$ be the natural projection, $i=1, 2$. 
Consider the varieties
\begin{align*}
\bold{Y}^s:=\{(l, (x, x^*&, i, j), (y, y^*))
\in \fp_{v_1, v_2}\times \bold{M}_{\overline{Q^{\heartsuit}}, v_1, w}^{ss}\times \bold{M}_{\overline Q, v_2}
\mid 
 [x, x^*]+i\circ j=\pi_1(l), [y, y^*]=\pi_2(l)\},\\
\bold{M}_{\widehat{Q}^{\spadesuit}, (v_1, v_2, w)}^{ss}
=&\bold{M}_{\widehat{Q}^{\spadesuit}, (v_1, v_2, w)}\cap 
(\bold{M}_{\overline{Q^{\heartsuit}}, (v_1, w)}^{ss}\times \fg_{v})\\
=&\{(l, x, x^*, i, j) \in \bold{M}_{\widehat{Q}^{\spadesuit}, (v_1, v_2, w)}\mid  l\in \fp_{v_1, v_2}, (x, x^*)\in \bold{M}_{\overline{Q}, (v, w)}^{ss} \}.
\end{align*}
Define a map from $q: \bold{M}_{\widehat{Q}^{\spadesuit}, (v_1, v_2, w)}^{ss} \to \fg_{v_1}\times \fg_{v_2}\times \bold{Y}^s$ by
$
(l, x, x^*, i, j) \mapsto$ $ (\pi_1(l), \pi_2(l), [x, x^*]+i\circ j, 
\pr_1(x), \pr_2(x), j_{V_1}, i_{V_1}).$
We have the embedding $\bold{\Lambda}_{v, w}\times \fg_{v}\subset \bold{M}_{\widehat{Q}^{\spadesuit}, (v, w)}$.
Let $\bold{\Lambda}_{v_1, v_2, w}^{ss}$ be the intersection of  $\bold{\Lambda}_{v, w}^{ss} \times \fg_{v}$ with 
$\bold{M}_{\widehat{Q}^{\spadesuit}, (v_1, v_2, w)}$. 

We have the following correspondence similar to \eqref{corr for action}, with the left square being Cartesian.
\begin{equation}\label{corr:the one in KS for mod}
\xymatrix@C=1.5em @R=1.5em{
{G\times_{P}(
\bold{\Lambda}_{v_1, w}^{ss} \times \fg_{v_1} 
\times \bold{\Lambda}_{v_2} \times \fg_{v_2}
)} \ar@{^{(}->}[d]& 
G\times_P(\bold{\Lambda}_{v_1, v_2, w}^{ss}\times \fp_{v_1, v_2})
\ar[l]_(0.45){\overline{q}}\ar[r]^(0.65){\overline {\hat\eta}} \ar@{^{(}->}[d]&
\bold{\Lambda}_{v, w}^{ss} \times 
\fg_{v}\ar@{^{(}->}[d]\\
G\times_{P}(\bold{Y}^{s} \times \fg_{v_1}\times \fg_{v_2}) & 
{G\times_P \bold{M}_{\widehat{Q}^{\spadesuit}, (v_1, v_2, w)}^{ss}}
\ar[l]_{q}\ar[r]^{\hat\eta} &
\bold{M}_{\overline{Q^{\heartsuit}}, (v, w)}^{ss} 
 \times \fg_{v}
}\end{equation}
The action map $a^{\aux}_{v_1, v_2}$ on $\mathcal{M}(w)=\bigoplus_{v\in \N^I}\calM(v, w)$ is defined to be the composition of the following morphisms.
\begin{enumerate}
\item The K\"unneth morphism
\begin{align*}
A_{G_{v_1}\times G_w \times T}(\bold{\Lambda}_{v_1, w}^{ss} \times \fg_{v_1})
\times 
A_{G_{v_2} \times T}(\bold{\Lambda}_{v_2} \times \fg_{v_2})
\to &
A_{L\times G_w \times T}(\bold{\Lambda}_{v_1, w}^{ss} \times \fg_{v_1} \times \bold{\Lambda}_{v_2} \times \fg_{v_2})\\
\cong &
A_{G\times G_w\times T} \big( G\times_{P} (\bold{\Lambda}_{v_1, w}^{ss} \times \fg_{v_1} \times \bold{\Lambda}_{v_2} \times \fg_{v_2})\big).
\end{align*}
\item 
The refined Gysin pullback $q^\sharp_{\overline{q}}$:
\begin{align*}
A_{G\times G_w\times T} \big( G\times_{P} (\bold{\Lambda}_{v_1, w}^{ss} \times \fg_{v_1} \times \bold{\Lambda}_{v_2} \times \fg_{v_2})\big) \to 
A_{G\times G_w\times T}
\big(G\times_P
(\bold{\Lambda}_{v_1, v_2, w}^{ss}\times \fp_{v_1, v_2}) \big).
\end{align*}
\item
The pushforward $\overline{\hat \eta}_*$ in the correspondence \eqref{corr:the one in KS for mod}
\begin{align*}
\overline{\hat \eta}_*:  A_{G\times G_w\times T}\big(G\times_P
(\bold{\Lambda}_{v_1, v_2, w}^{ss}\times \fp_{v_1, v_2})
\big)\to 
A_{G\times G_w\times T}\big(
\bold{\Lambda}_{v_1+v_2, w}^{ss} \times 
\fg_{v}\big).
\end{align*}
\end{enumerate}
We have the following
\begin{theorem}\label{thm:crit CoHA action}
For any $w\in\bbN^I$, the maps $a^{\aux}_{v_1,v_2}$  define an algebra homomorphism 
$\calH^{\aux}(\widehat{Q}, W, C) \to \End(\calM(w))$.
\end{theorem}
The proof of Theorem~\ref{thm:crit CoHA action} goes the same way as that of \cite[Theorem~5.4]{YZ15}, taking into account the Lie algebra factors in the spaces which do not show up in {\it loc. cit.}.

\subsection{Multiplications of $\calH^{\aux}(\widehat{Q}, W, C)$ and $\calP(A, Q)$}
\label{sec:aux and pro projective}
By the fact that $A_{G\times T}(\bold{\Lambda}_{v}\times \fg_{v})\cong A_{G\times T}(\bold{\Lambda}_{v})$, there is an isomorphism of vector spaces $\calP_v(Q)\cong \calH_v^{\aux}(\widehat{Q}, W, C)$ for any $v\in\bbN^I$. However, this isomorphism is not compatible with the multiplications.
\begin{prop}\label{prop:prep and crit}
Let $m^{\aux}$ be the multiplication of  $\calH^{\aux}(\widehat{Q}, W, C)$. Then, for $x\in\calH_{v_1}^{\aux}$ and $y\in\calH_{v_2}^{\aux}$, we have
\[
m^{\aux}(x\otimes y)=\overline \psi_{*} \Big(e(\gamma) \cdot \phi^{\sharp}_{\overline{\phi}} (x\otimes y)\Big),
\]
where $\phi$ is as in \S~\ref{subsec:HallMulti}, \eqref{diag:Lag}
and $
e(\gamma)$ is the equivariant Euler class of the closed embedding
\begin{align*}
\gamma: G\times_P
(\bold{\Lambda}_{v_1, v_2}\times \fp_{v_1, v_2} )
&\inj (G\times_P
\bold{\Lambda}_{v_1, v_2})
\times \fg_{v_1+v_2},\\
(g, a, b) &\mapsto ((g, a), gbg^{-1}), \,\ \text{for $a\in \bold{\Lambda}_{v_1+v_2}, b\in \fp_{v_1, v_2}$.}
\end{align*}
\end{prop}

\begin{proof}
Recall the following diagram of correspondences:
\[
\xymatrix@C=1.5em @R=1.5em{
G\times_{P}(\bold{\Lambda}_{v_1}\times \bold{\Lambda}_{v_2}) & G\times_P \bold{\Lambda}_{v_1, v_2}\ar[l]_{\overline \phi}\ar[r]^{\overline \psi} &\bold{\Lambda}_{v_1+v_2}\\
G\times_{P} (
\bold{\Lambda}_{v_1} \times \fg_{v_1} 
\times \bold{\Lambda}_{v_2} \times \fg_{v_2}
)\ar[u]^{\pi_{v_1}\times \pi_{v_2}}& 
G\times_P (\bold{\Lambda}_{v_1, v_2}\times \fp_{v_1, v_2})
\ar[l]_(0.4){\overline q}\ar[r]^(0.6){\overline {\hat\eta}} \ar[u]^{\pi}  &
\bold{\Lambda}_{v_1+v_2} \times \fg_{v}\ar[u]_{\pi_{v_1+v_2}}
}\]
where the map $\pi$ is a vector bundle with fiber $\fp_{v_1, v_2}$.
Under the natural isomorphisms induced by affine bundles, we have 
$q^{\sharp}_{\overline{q}}=\phi^{\sharp}_{\overline \phi}$.
Note that the right square is not a Cartesian diagram. We factor the map $\overline {\hat\eta}$ as $(\overline{\psi}\times \id)\circ \gamma$, where $\gamma$ is a closed embedding. 
 \begin{equation}\label{dia:gamma} \xymatrix@R=1.5em{
&G\times_P \bold{\Lambda}_{v_1, v_2}\ar[r]^{\overline \psi} &\bold{\Lambda}_{v_1+v_2}\\
G\times_P (\bold{\Lambda}_{v_1, v_2}\times \fp_{v_1, v_2})
\ar@{^{(}->}[r]^{\gamma}\ar[ru]^{\pi}
\ar@/_1.5pc/[rr]_{\overline{\hat \eta}}&
(G\times_P \bold{\Lambda}_{v_1, v_2}) \times \fg_{v_1+v_2}
\ar[r]^(0.55){\overline{\psi}\times \id}\ar[u]^{\pr_1}  &
\bold{\Lambda}_{v_1+v_2} \times 
\fg_{v_1+v_2}\ar[u]_{\pi_{v_1+v_2}}
}\end{equation}
Clearly, the square in above diagram is a pullback diagram and satisfies the conditions in \cite[Lemma~1.16]{YZ15}. The pushforward $\gamma_*: A_{G\times T}(G\times_P \bold{\Lambda}_{v_1, v_2}) \to A_{G\times T}(G\times_P \bold{\Lambda}_{v_1, v_2})$ is given by $
 x\mapsto x\cdot e(\gamma)$. The proposition now follows from the definition of Hall multiplication $m^{\aux}$ of $\calH^{\aux}(\widehat{Q}, W, C)$.
\end{proof}
\Omit{
Proposition \ref{prop:prep and crit} and the shuffle formula \eqref{shuffle formula} give the following shuffle description of $\calH^{\aux}(\widehat{Q}, W, C)$.
Let $\calS\calH^{\aux}$ be the auxiliary shuffle algebra, which is isomorphic to the shuffle algebra 
$\calS\calH$ as abelian groups. 
For $f_1\in \calS \calH_{v_1}^{\aux}$ and $f_2\in \calS\calH_{v_2}^{\aux}$, the multiplication $m_{v_1,v_2}^{\aux}(f_1\otimes f_2)$ is given by the shuffle formula
\begin{equation}\label{equ:crit shuffle}
\sum_{\sigma\in \Sh(v_1,v_2)}\sigma\big(f_1(\lambda'^i_s)\cdot f_2(\lambda''^j_t )
\cdot \fac_1\cdot \fac_2\cdot e(\gamma) \big), \end{equation}
where $\fac_1, \fac_2$ are given by \eqref{equ:fac1} and \eqref{equ:fac2}. 
}
Recall that $\gamma$ is the map
\[
\gamma: G\times_P (\fp_{v_1, v_2}\times \bold{\Lambda}_{v_1, v_2}) \inj 
G\times_P \bold{\Lambda}_{v_1, v_2} \times \fg_{v_1+v_2}, \,\ 
(g, a, b) \mapsto ((g, a), gbg^{-1}).\] 
In particular 
the normal bundle to $\gamma$ can be identified with the normal bundle of $\fp_{v_1, v_2}$ in $\fg_{v_1+v_2}$. 

Let $a^{\prepr}$ be the (right) action of $\calP(A, Q)$ on $\calM(w)$. By construction, $a^{\prepr}(m\otimes x)=\overline \psi_{*}  \phi^{\sharp}_{\overline{\phi}} (m\otimes x)$, where $\phi$ and $\overline{\psi}$ are as in \S~\ref{subsec:PreproRepn}, \eqref{corr for action}. 
\begin{prop}\label{prop:prep and crit for action}
Let $a^{\aux}$ be the (right) action of $\calH^{\aux}(\widehat{Q}, W, C)$ on $\calM(w)$.
We then have:
\[
a^{\aux}(m\otimes x)=\overline \psi_{*} \big(e(\gamma) \cdot \phi^{\sharp}_{\overline{\phi}} (m\otimes x)\big).
\]
\end{prop}

\begin{proof}
The proof is similar as the proof of Proposition \ref{prop:prep and crit}. 
\end{proof}

Recall that $m^{\prepr}(x\otimes y)=\overline \psi_{*}  \phi^{\sharp}_{\overline{\phi}} (x\otimes y)$, $x\in\calP_{v_1}$ and $y\in\calP_{v_2}$.
Proposition \ref{prop:prep and crit} compares the multiplications of $\calP(A, Q)$ and $\calH^{\aux}(\widehat{Q}, W, C)$, but it does not give an algebra homomorphism between the two algebras.
\Omit{
In Appendix~\ref{app:sign-twist}, we interpret Proposition~\ref{prop:prep and crit} and \ref{prop:prep and crit for action} in terms of an algebra homomorphism from $\calP(Q)$ to a sign-twisted version of $\calH^{\aux}(\widehat{Q}, W, C)$.
}

\section{Conclusions}\label{sec:concl}
In this section, we compare the preprojective CoHA of $Q$ and the critical CoHA defined by Kontsevich-Soibelman  associated to the quiver with potential $(\widehat{Q},W)$, hence prove Theorem \ref{ThmIntr:Crit}. By Theorem~\ref{prop: crit CoHA with the one in KS}, it suffices to compare the preprojective CoHA with $\calH^{\crit}(\widehat{Q}, W, C)=\calH^{\aux}(\widehat{Q}, W, C)$, with the multiplication $m^{\crit}= \overline{i_{2}}_* \circ \frac{1}{e(\iota)}\omega_{\overline{\omega}}^{\sharp} \circ \overline{i_{1}}_* \circ \overline{p_1}^{*}$.

In this section we take $A$ to be the Borel-Moore homology. In Remark~\ref{rem:BM} we comment on the generality in which the main theorem of this section holds.
\subsection{The main theorem}
Recall that in \eqref{corresp with fiber}, we have the following map
\[\iota: G\times_{P}(\bold{M}_{\overline{Q}, v_1}\times \bold{M}_{\overline{Q}, v_2})
\inj
G\times_{P}\bold{Y}. \] 
Let $e(\iota)$ be the equivariant Euler class of the normal bundle of $\iota$. By \cite[\S~3.2]{YZ15}, the normal bundle of $\iota$ is isomorphic to $T^*G/P$ as a bundle over the Grassmannian $G/P$, which equivariantly is $G\times_P(\fg_v/\fp)^*$. Here the $T$-action on $\bold{M}_{C,v}\cong \fg_v$ is specified in \S~\ref{preproj CoHA}. Recall that the normal bundle of $\gamma$ is the normal bundle of $\fp_{v_1, v_2}$ in $\fg_{v_1+v_2}$, with the $T$-action induced from that on $\fg$.
Therefore, we have
 \[e(\gamma)=\prod_{i\in I}(-1)^{v_1^iv_2^i}e(\iota) =(-1)^{v_1v_2}e(\iota), \]
where $(-1)^{v_1\cdot v_2}:=\prod_{i\in I}(-1)^{v_1^iv_2^i}$ for any $v_1$, $v_2\in\bbN^I$. 

We have the main theorem of this section.
\begin{theorem}\label{thm:KS_prep}
\begin{enumerate} 
\item
Under Assumption~\ref{Assu:WeghtsGeneral}, there is an isomorphism of $\bbN^I$-graded associative  algebras $\Xi: \calP(\BM, Q)\to\calH^{\crit}(\widehat{Q}, W)$ whose restriction to the degree-$v$ piece is
\[
\Xi_v: \calP(\BM, Q)_{v}\to  \calH^{\crit}(\widehat{Q}, W)_{v}, \,\ \,\
f\mapsto f\cdot (-1)^{{v}\choose{2}},
\] where $(-1)^{{v}\choose{2}}:=\prod_{i\in I}(-1)^{{v^i}\choose{2}}$.  Here we define ${n\choose 2}=0$ if $n=0, 1$. 
\item
In the setup of (1), 
we have
\[
a^{\crit}\big(\Xi_{v_1}(x)\big)\big(m\cdot (-1)^{{v_2}\choose{2}}\big)=\big(a^{\prepr}(x)(m)\big)\cdot (-1)^{{v_1+v_2}\choose{2}})
\]
for any $w,v_1,v_2\in\bbN^I$, $x\in \calP_{v_1}$, and $m\in H^{\BM}_{G_w\times T}(\mathfrak{M}(v_2,w))$.
\end{enumerate}
\end{theorem}

\begin{proof}
We only prove (1). Part (2) follows from a similar argument. By  \cite[Corollary A.9]{D} (recalled below as Theorem~A.1), $\calP(\BM, Q)$ is isomorphic to $\calH^{\crit}(\widehat{Q}, W)$ as abelian groups. It suffices to show the map $\Xi_v$ respects the multiplication structure. 

For $x\in\calP_{v_1}$ and $y\in\calP_{v_2}$, by definition, $m^{\prepr}(x\otimes y)=\overline \psi_{*}  \phi^{\sharp}_{\overline{\phi}} (x\otimes y)$, where $\phi$ and  $\overline \psi$ are as in \S~\ref{subsec:HallMulti} \eqref{diag:Lag}. On $\calH^{\aux}(\widehat{Q}, W, C)=\calP(\BM,Q)$,  
by Proposition \ref{prop:prep and crit} we have
\[
m^{\aux}(x\otimes y)=\overline \psi_{*} \Big(e(\gamma) \cdot \phi^{\sharp}_{\overline{\phi}} (x\otimes y)\Big)=
(-1)^{v_1v_2} \overline \psi_{*} \Big(e(\iota) \cdot \phi^{\sharp}_{\overline{\phi}} (x\otimes y)\Big)
.\]
On the other hand, on $\calH^{\crit}(\widehat{Q}, W, C)=\calH^{\aux}(\widehat{Q}, W, C)$, by Theorem \ref{prop: crit CoHA with the one in KS} we have
$m^{\crit}=\overline{i_{2}}_* \circ \frac{1}{e(\iota)}\omega_{\overline{\omega}}^{\sharp} \circ  \overline{i_{1}}_* \circ \overline{p_1}^{*}$, while $m^{\aux}=\overline{i_{2}}_* \circ \omega_{\overline{\omega}}^{\sharp} \circ  \overline{i_{1}}_* \circ \overline{p_1}^{*}$. Therefore,  $m^{\prepr}(x\otimes y)= (-1)^{v_1v_2} m^{\crit} (x\otimes y)$. 

Now we verify that the isomorphism of vector spaces $\Xi: (\calP(\BM, Q), m^{\prepr})\to (\calH^{\crit}(\widehat{Q}, W, C), m^{\crit})$ is an algebra homomorphism. Indeed, on the one hand, we have 
\[
\Xi( m^{\prepr}(x\otimes y))= \Xi((-1)^{v_1v_2} m^{\crit} (x\otimes y))=(-1)^{ {v_1+v_2}\choose {2}}(-1)^{v_1v_2} m^{\crit} (x\otimes y).\] On the other hand, we have
\[
m^{\crit}( \Xi(x)\otimes \Xi(y))
=m^{\crit}( (-1)^{{v_1}\choose{2} }(x)\otimes (-1)^{{v_2}\choose{2} }(y))
=(-1)^{{{v_1}\choose{2}}} (-1)^{{v_2}\choose{2}} m^{\crit}(x\otimes y). \]
The equality ${v_1^i \choose 2}+{v_2^i \choose 2}+v_1^i v_2^i={v_1^i+v_2^i\choose 2}$, for any $i\in I$, shows $\Xi( m^{\prepr}(x\otimes y))=m^{\crit}( \Xi(x)\otimes \Xi(y))$. Thus, $\Xi$ is an algebra homomorphism. 
\end{proof}

\begin{remark}\label{rem:BM}
Let $A$ be an arbitrary oriented cohomology theory. Let $(R,F)$ be the formal group law associated to $A$. Then, Theorem~\ref{thm:KS_prep} holds as long as $F$ is an odd function in the sense that $x-_Fy=-(y-_Fx)$. 
\end{remark}

\appendix
\section{Borel-Moore homology and critical cohomology}\label{App}
In this section, we show the compatibility of push-forwards and pull-backs in the Borel-Moore homology and the critical cohomology. 

\subsection{From the critical cohomology to ordinary cohomology}\label{subsec:app_reduc}
We compare the critical cohomology with the ordinary cohomology in this section, following the appendix of \cite{D}.
Let $\pi: Y=X\times  \mathbb{A}^n \to X$ be the trivial vector bundle, carrying a scaling $\Gm$ action on the fiber $ \mathbb{A}^n$.
Let $f: Y=X\times  \mathbb{A}^n\to \mathbb{A}^1$ be a $\Gm$--equivariant function with respect to  the natural scaling $\Gm$ action on the target.
For simplicity we assume $f^{-1}(0)\supset\Crit(f)$.
Define $Z \subset X$ to be the reduced scheme consisting of points $z\in X$, such that $\pi^{-1}(z)\subset f^{-1}(0)$. 
To summarize the notations, we have the diagram:
\[
\xymatrix@R=1.5em{
Z \times  \mathbb{A}^n \ar@{^{(}->}[r]^{i\times \id}\ar[d]^{\pi_Z} &X\times  \mathbb{A}^n \ar[d]^{\pi} \ar[dr]^{p_Y}&\\
Z \ar@{^{(}->}[r]^{i} &X  \ar[r]^{p}& \pt.
}\]

Let $\varphi_{f}$ be the  vanishing cycle functor for $f$. Following the convention of \cite{D},
we consider $\varphi_{f}$ as a functor $D^b(Y)\to D^b(Y)$ between the derived categories of $Y=X\times \mathbb{A}^n$.
By an abuse of notation, we will abbreviate the vanishing cycle complex $\varphi_{f} \Q_Y[-1]$ to $\varphi_{f}$. The support of $\varphi_{f}$ is on the critical locus of $f$. For a $G$-variety $X$, let $H_{c, G}^*(X)$ be the equivariant cohomology with compact support. Let 
 $H_{c, G_v}^*(X)^\vee$ be its Verdier dual.

\begin{theorem}[\cite{D}, Theorem A.1 and Corollary A.9]\label{thm:DavA1}
There is a natural isomorphism of functors $D^{b}(X) \to D^{b}(X),$
$
\pi_{!} \varphi_f \pi^*[-1]\cong \pi_{!} \pi^* i_* i^*.
$ In particular, we have
$
H_{c, G}^*(Y, \varphi_f)\cong H^*_{c, G}(Z\times \mathbb{A}^n, \Q).$
\end{theorem}
Indeed, by definition, we have
\[
H_{c}^*(Y, \varphi_f)=p_{Y!} \varphi_f[-1](\Q_Y)
                                   =p_{!} \pi_{!} \varphi_f [-1](\pi^* (\Q_X)). 
\] On the other hand, we have the isomorphism
\begin{align*}
H_{c}^*(Z\times \mathbb{A}^n, \Q)
=&p_{!}\pi_{!} ( i\times \id)_{!}\Q_{Z\times \mathbb{A}^n}
=p_{!}\pi_{!} ( i\times \id)_{*} \pi_{Z}^*\Q_{Z}
=p_{!}\pi_{!}  \pi_{}^* i_* i^*\Q_{X}.
\end{align*}
Thus, the isomorphism $H_{c, G}^*(Y, \varphi_f)\cong H^*_{c, G}(Z\times \mathbb{A}^n, \Q)$ follows from the isomorphism of the two functors in Theorem \ref{thm:DavA1}, which is shown in \cite[Theorem A.1]{D}. 
\begin{prop}\cite[Proposition A.8]{D}\label{Ts and Ku}
The following diagram of isomorphisms commutes.
\[\xymatrix@R=1.5em{
H^*_c(f_1^{-1}(0), \varphi_{f_1})\otimes H^*_c(f_2^{-1}(0), \varphi_{f_2}) \ar[r]^(0.55){\hbox{TS}}\ar[d]^{\cong} & 
H^*_c(f_1^{-1}(0)\times f_2^{-1}(0), \varphi_{f_1\boxplus f_2}) \ar[d]^{\cong}\\
H^*_c(Z_1\times \mathbb{A}^{n_1}, \Q)\otimes 
H^*_c(Z_2\times \mathbb{A}^{n_2}, \Q)\ar[r]^(0.55){\hbox{Ku}} &
H^*_c(Z_1\times Z_2\times \mathbb{A}^{n_1+n_2}, \Q),
}\]
where $\hbox{TS}$ is the Thom-Sebastiani isomorphism, $\hbox{Ku}$ is the K\"unneth isomorphism, and the vertical isomorphisms are as in Theorem A.1 of \cite{D}.
\end{prop}

\subsection{Compatibility of push-forwards and pullbacks}
In this section, we show the isomorphism in Theorem \ref{thm:DavA1}
is compatible with  pullbacks and  proper pushforwards. 

Let $g: X\to X'$ be a morphism and $g\times h: Y=X\times \mathbb{A}^n\to Y'=X'\times \mathbb{A}^m$ be the morphism of the trivial bundles,
where $h: \mathbb{A}^n\to \mathbb{A}^m$ is a linear morphism. 
Let $f': X'\times \mathbb{A}^m\to \mathbb{A}^1$ be a function, and $f:=f'\circ (g\times h)$. 

\begin{lemma}\label{lem:fiber_diag}
Assume $h$ is a surjective linear map, then there is a map $g_Z:Z\to Z'$ induced by $g:X\to X'$ so that $Z$ is the fiber product of $X$ and $Z'$, i.e., the bottom square in the following diagram is Cartesian. In particular, in this case the top square in the following diagram is also Cartesian.
\end{lemma}
\begin{equation}\label{diagram: 3D}
\xymatrix@C=1em @R=1.5em {
Z \times  \mathbb{A}^n \ar@{^{(}->}[rr]^{i\times \id}\ar[dd]^{\pi_Z} 
\ar[rd]^{g_Z \times h}
&&X\times  \mathbb{A}^n \ar@{-->}[dd]^(0.3){\pi_X} \ar[rd]^{g \times h}&\\
&Z' \times  \mathbb{A}^m \ar@{^{(}->}[rr]^(0.3){i'\times \id}\ar[dd]_(0.3){\pi_{Z'}} &&X'\times  \mathbb{A}^{m} \ar[dd]^{\pi_{X'}} \\
Z \ar@{^{(}-->}[rr]^(0.3){i} \ar[rd]_{g_Z}& &
 X \ar@{-->}[rd]^{g}& \\
&Z' \ar@{^{(}->}[rr]_{i'} & & X'. 
}
\end{equation}
\begin{proof}
This follows from the definitions of $Z$ and $Z'$, and the assumption that $h$ is surjective.
\end{proof}
For the purpose of the present paper we only consider cases in which $h$ is surjective. 
\begin{lemma}\label{lem:pushforward}
With notations as above, assume $g$ is a proper morphism,  $m=n$ and $h$ is the identity map.
Then, the following diagram commutes.
\[
\xymatrix@R=1.5em{
H_{c}^*(X\times \mathbb{A}^n, \varphi_f)^\vee
\ar[r]^{(g\times h)_*}\ar[d]_{\cong} & 
H_{c}^*(X'\times \mathbb{A}^m, \varphi_{f'})^\vee
 \ar[d]_{\cong}\\ 
H_{c}^*(Z\times \mathbb{A}^n, \Q)^\vee
\ar[r]^{(g_Z\times h)_{*}} & 
H_{c}^*(Z'\times \mathbb{A}^m, \Q)^\vee.
}\]
In the diagram, the vertical isomorphisms are given in Theorem \ref{thm:DavA1}.
\end{lemma}
\begin{proof}
By the commutativity of vanishing cycle functors with proper pushforwards,
the commutativity of the diagram in the lemma is equivalent to the commutativity of the following diagram:
\[
\xymatrix@R=1.5em{
&p_{X'\times \mathbb{A}^{m}!} 
\varphi_{f'} [-1]\Big(\Q_{X'\times \mathbb{A}^{m} } \ar[r]
\ar[d]^{\cong}
& (g\times h)_* \Q_{X\times \mathbb{A}^{n} } \Big) \ar[d]^{\cong}\\
&p_{X'\times \mathbb{A}^{m} !} (i'\times \id)_{*} (i'\times \id)^{*}
\Big( \Q_{X'\times \mathbb{A}^{m}} \ar[r] &(g\times h)_* \Q_{X\times \mathbb{A}^{n}}\Big).
}\]
\Omit{
Indeed, the meaning of $(g\times h)_*$ on the top (see page 17 of \cite{D}) is
$p_{X'\times \mathbb{A}^{m}!} 
\varphi_{f}\Big(\Q_{X'\times \mathbb{A}^{m} } \to (g\times h)_* \Q_{X\times \mathbb{A}^{n} } \Big)$, and the meaning of $(g_Z\times h)_*$ on the bottom is 
\begin{align*}
&p_{Z'\times \mathbb{A}^{m} !}\Big( \Q_{Z'\times \mathbb{A}^{m}} \to (g_Z\times h)_* \Q_{Z\times \mathbb{A}^{n}}\Big)\\
=&p_{X'\times \mathbb{A}^{m} !} (i'\times \id)_{*} (i'\times \id)^{*}
\Big( \Q_{X'\times \mathbb{A}^{m}} \to (g\times h)_* \Q_{X\times \mathbb{A}^{n}}\Big).
\end{align*}
}

Applying the two functors 
\[F=p_{X'!}(\pi_{X'})_{!} \varphi_{f'} (\pi_{X'})^*[-1], \,\
G=p_{X'!}(\pi_{X'})_{!} (\pi_{X'})^*i'_* i'^*\]
to the morphism $\big( \Q_{X'} \to g_* \Q_{X}\big)$ gives the desired commutativity.
\end{proof}

Now we consider pullbacks. 
\begin{lemma}\label{lem:pullback}
With notations as in diagram \eqref{diagram: 3D}, 
then, the following diagram, in which the vertical isomorphisms are given in Theorem~ \ref{thm:DavA1}, commutes in the following two cases
\begin{enumerate}
\item $m=n$, and $h$ is the identity map;
\item $h$ is a surjective linear map, $X'=X$ and $g:X\to X'$ is the identity map.
\end{enumerate}
\[
\xymatrix@R=1.5em{
H_{c}^*(X'\times \mathbb{A}^m, \varphi_{f'})
^{\vee}
\ar[r]^(0.55){(g\times h)^*}\ar[d]_{\cong} & 
H_{c}^*(X\times \mathbb{A}^n, \varphi_f)^{\vee}
 \ar[d]_{\cong}\\ 
H_{c}^*(Z'\times \mathbb{A}^m, \Q)^{\vee}
\ar[r]^(0.55){(g\times h)^{*}} & 
H_{c}^*(Z\times \mathbb{A}^n, \Q)^{\vee}.
}
\]
\end{lemma}
\begin{proof}
In the case $m=n$, and $h$ is the identity map, the commutativity is equivalent to
the commutativity of the following diagram:
\[
\xymatrix@C=1em@R=1.5em {
p_{X'\times \mathbb{A}^m !}\varphi_{f'} [-1]\Big( (g\times h)_{!} \Q_{X\times \mathbb{A}^n} \ar[r] \ar[d]& \Q_{X'\times \mathbb{A}^m}[\dim(g\times h)] \Big)\ar[d]\\
p_{X'\times \mathbb{A}^m !}(i'\times \id)_{*} (i'\times \id)^{*}
\Big( (g\times h)_{!} 
\Q_{X\times \mathbb{A}^n} \ar[r] &\Q_{X'\times \mathbb{A}^m}[\dim(g\times h)] \Big).
}\]
Applying the two functors
\[
F=p_{X'\times \mathbb{A}^n !}  \varphi_{f'} \pi_{X'}^*[-1], \,\
G=p_{X'\times \mathbb{A}^n !} \pi_{X'}^* i'_* i'^*\]
to the morphism $\big(g_{!} \Q_{X} \to \Q_{X'}[\dim g] \big)$ gives  the desired commutativity. 

In the case when $g$ is the identity map, by Lemma~\ref{lem:fiber_diag}, the induced map $g_Z$ is an isomorphism. Both the top and the bottom of the diagram become the natural isomorphism induced by an affine bundle.
\end{proof}

\newcommand{\arxiv}[1]
{\texttt{\href{http://arxiv.org/abs/#1}{arXiv:#1}}}
\newcommand{\doi}[1]
{\texttt{\href{http://dx.doi.org/#1}{doi:#1}}}
\renewcommand{\MR}[1]
{\href{http://www.ams.org/mathscinet-getitem?mr=#1}{MR#1}}

\end{document}